\documentclass[11pt]{article}

\usepackage[margin=1in]{geometry}
\usepackage{amsthm,amsmath,amsfonts,amssymb,mathtools}
\usepackage{aliascnt}
\usepackage[numbers]{natbib}
\usepackage[colorlinks,citecolor=blue,urlcolor=blue]{hyperref}
\usepackage[nameinlink,capitalize]{cleveref}
\usepackage{enumitem}
\usepackage{microtype}

\newcommand{\BetaIncidenceAuthorAffiliation}{Department of Statistics and
  Operations Research, University of North Carolina at Chapel Hill}

\newcommand{\BetaIncidenceAuthorEmail}{tianletl@unc.edu}

\newtheoremstyle{thmstyleone}%
  {6pt}{6pt}{\itshape}{}{}{.}{0.5em}{}
\newtheoremstyle{thmstyletwo}%
  {6pt}{6pt}{\normalfont}{}{}{.}{0.5em}{}

\numberwithin{equation}{section}
\theoremstyle{thmstyleone}%
\newtheorem{theorem}{Theorem}[section]
\newaliascnt{proposition}{theorem}
\newtheorem{proposition}[proposition]{Proposition}
\aliascntresetthe{proposition}
\newaliascnt{lemma}{theorem}
\newtheorem{lemma}[lemma]{Lemma}
\aliascntresetthe{lemma}
\newaliascnt{definition}{theorem}

\aliascntresetthe{definition}
\newaliascnt{corollary}{theorem}
\newtheorem{corollary}[corollary]{Corollary}
\aliascntresetthe{corollary}
\newaliascnt{conjecture}{theorem}

\aliascntresetthe{conjecture}
\theoremstyle{thmstyletwo}%
\newaliascnt{example}{theorem}
\newtheorem{example}[example]{Example}
\aliascntresetthe{example}
\newtheorem*{remark}{Remark}%

\makeatletter
\@ifundefined{acks}{%
  \newenvironment{acks}[1][Acknowledgments]{\section*{#1}}{\par}
}{}
\makeatother

\newcommand{\R}{\mathbb R}

\newcommand{\D}{\mathbb D}
\newcommand{\E}{\mathbb E}
\newcommand{\Prb}{\mathbb P}
\newcommand{\sphere}{\mathbb S}

\newcommand{\dd}{\,\mathrm d}
\newcommand{\Beta}{\operatorname{Beta}}
\newcommand{\BetaPrime}{\operatorname{BetaPrime}}
\newcommand{\Dir}{\operatorname{Dirichlet}}
\newcommand{\Cauchy}{\operatorname{Cauchy}}

\newcommand{\Unif}{\operatorname{Unif}}
\newcommand{\dist}{\operatorname{dist}}
\newcommand{\Law}{\mathcal L}
 \newcommand{\BetaIncidenceBibliographyStyle}{alpha}

\hypersetup{
  pdftitle={Universal Beta Incidence Angles: Cauchy Rigidity and Infinite Arrangements},
  pdfauthor={Tianle Liu}
}

\title{Universal Beta Incidence Angles: Cauchy Rigidity and\\
Infinite Arrangements}
\author{Tianle Liu\\
  \small\BetaIncidenceAuthorAffiliation\\
  \small\href{mailto:\BetaIncidenceAuthorEmail}
    {\texttt{\BetaIncidenceAuthorEmail}}}
\date{August 2026}

\begin{document}

\maketitle

\begin{abstract}
Let \(U\) be Haar-uniform on the unit sphere, and let a weighted central
hyperplane arrangement define
\[
  g(U)=\sum_j w_j\frac{a_j}{a_j^\top U},
  \qquad
  N(U)=\frac{g(U)}{\|g(U)\|}.
\]
Although \(N(U)\) is a deterministic, generally non-Haar function of \(U\),
we prove that its squared incidence cosine with \(U\) has the same beta law
as the squared inner product of two independent Haar directions.  One proof
combines the classical Herglotz--Cauchy boundary principle, a Haar-random
two-plane with a common phase, and an exact beta--Cauchy Mellin equivalence;
a second specializes the positive-semidefinite Pillai--Meng identity.

Beyond this marginal law, plane-conditional Cauchy identities recover
positivity for finite reciprocal arrangements.  For signed measures, an
exact phase-cancellation deficit equals twice the hidden negative mass,
yielding local-to-global positivity under a phase-norming condition weaker
than injectivity, a cross-sign collision criterion, and exclusion of
negative atoms.  The simplex is the maximal normalized real-weight class
for finite universality.

The law extends to probability measures under reciprocal integrability,
which we characterize by a Wiener--Dini belt series.  Finite Shannon entropy
is the sharp universal criterion for countable weights, and an
entropy--geometry theorem treats clustered measures.  Every compact carrier
of zero \(\mathcal H^1\)-measure is admissible, whereas a nonzero rectifiable
arc component forces divergence on a set of positive Haar measure.  In
orthogonal coordinates, the theorem also yields an exact scaled \(F\) law
for Pearson divergence from a fixed simplex vector to a
\(\Dir(1/2,\ldots,1/2)\) vector.
 \end{abstract}



\section{Introduction}
\label{sec:introduction}

Let \(a_1,\ldots,a_k\) be arbitrary nonzero vectors in \(\R^p\), let
\(w_1,\ldots,w_k\) be nonnegative weights summing to one, and let
\(U\) be Haar-uniform on \(\sphere^{p-1}\).  Away from the arrangement
hyperplanes, consider
\[
  g_{a,w}(U)=\sum_{j=1}^k w_j\frac{a_j}{a_j^\top U},
  \qquad
  N_{a,w}(U)=\frac{g_{a,w}(U)}{\|g_{a,w}(U)\|}.
\]
The weight normalization gives \(U^\top g_{a,w}(U)=1\).  The angle between \(U\)
and \(N_{a,w}(U)\) therefore lies in \([0,\pi/2)\) and satisfies
\[
  \{U^\top N_{a,w}(U)\}^2=\frac{1}{\|g_{a,w}(U)\|^2}.
\]
Our starting phenomenon is that
\begin{equation}
  \{U^\top N_{a,w}(U)\}^2
  \sim
  \Beta\!\left(\frac12,\frac{p-1}{2}\right),
  \label{eq:intro-beta-law}
\end{equation}
independently of the number, positions, rank, or overcompleteness of the
directions and independently of the simplex weights.  The law in
\cref{eq:intro-beta-law} is exactly the squared-cosine law for two independent
Haar directions.  Thus a deterministic, generally non-Haar function of
\(U\) mimics independence through one scalar incidence observable.

There are two routes to \cref{eq:intro-beta-law}.  In the planar route, a
Herglotz boundary theorem proves the
weighted shifted-tangent Cauchy identity; a Haar-random two-plane supplies
one common uniform phase; and an exact Mellin equivalence lifts the Cauchy
projection back to the spherical Beta angle.  The Gaussian route then
specializes the Pillai--Meng Cauchy--L\'evy theorem
\cite{pillai-meng-2016}: choose the possibly singular Gram covariance
\(A^\top A\) and apply Gaussian polar decomposition.  The Pillai--Meng proof
itself uses the same weighted shifted-tangent identity after extracting a
common angular phase.  The planar derivation does not invoke the
Pillai--Meng theorem, while the two routes share this classical
one-dimensional Cauchy-preserver mechanism.  We present both before
developing the converse and infinite-arrangement theory.

The Pillai--Meng theorem contains the finite distributional law through its
positive-semidefinite covariance formulation.  The incidence viewpoint
identifies the associated Euler-normalized normal field and the
independent-Haar squared-cosine law realized by its deterministic coupling
with \(U\).  The planar route also reveals the conditional structure used
below for the converse and infinite-arrangement results.

The organizing principle is a hierarchy of information.  The unconditional
Beta marginal fixes only a tangent norm.  Plane-conditional Cauchy laws
recover substantially more reciprocal structure.  An oriented analytic
condition recovers the sign information that a symmetric distribution
necessarily loses.  Counterexamples at the first two transitions show why
each strengthening is needed.

The main results determine the scope of this phenomenon.

\begin{enumerate}[label=\textup{(\roman*)}]
\item For every measurable tangent field, the Beta law for its tangent norm
  is equivalent to a standard Cauchy projection onto one conditionally Haar
  tangent direction.  An explicit swirl field shows that this scalar
  equivalence does not characterize gradients or common phase.
\item A signed Herglotz tail theorem makes positivity equivalent to a
  standard Cauchy boundary law.  This yields conditional rigidity for each
  fixed finite arrangement.  In a nonlinear zonal class, an Abel-transform
  argument identifies every admissible magnitude and isolates the
  unavoidable measurable sign and radial-dual obstructions.
\item The incidence law extends to Borel probability measures whenever the
  reciprocal vector field is absolutely integrable almost surely.
  Moreover, this integrability forces almost every planar phase projection
  to be singular and recovers the conditional Cauchy mechanism.
\item For signed measures, the conditional laws make every planar
  pushforward positive and identify an exact phase-cancellation deficit equal
  to twice the hidden negative mass.  Positivity follows under a phase-norming
  condition that is strictly weaker than injectivity; a cross-sign collision
  theorem and a no-negative-atoms theorem give concrete consequences.
\item An exact Wiener--Dini series characterizes reciprocal integrability
  pointwise.  For countable positive arrangements, finite Shannon entropy is
  necessary and sufficient uniformly over deterministic directions, while
  an entropy--geometry gluing theorem covers clustered measures beyond that
  criterion.
\item At the critical geometric threshold, every compact carrier of zero
  one-dimensional Hausdorff measure is admissible, including
  dimension-one examples, whereas any nonzero rectifiable arc component
  forces divergence on a set of positive Haar measure.
\end{enumerate}

One coordinate consequence is especially easy to state.  If
\(S\sim\Dir(1/2,\ldots,1/2)\), then for every fixed simplex vector \(w\),
the Pearson divergence satisfies
\[
  D_{\chi^2}(w\Vert S)\overset d=(p-1)F_{p-1,1}.
\]
Thus the incidence theorem also produces a weight-free exact law in
classical distribution theory; see \cref{cor:pearson-divergence}.

The arrangement master functions and logarithmic one-forms underlying the
construction are classical \cite{cohen-denham-falk-varchenko-2012}.
Related logarithmic incidence potentials and log-cosine representations
occur in \cite{hejhal-2004,kalton-koldobsky-yaskin-yaskina-2007}, while
generalized cosine transforms provide a wider integral-geometric context
\cite{rubin-2008}.
The analytic ingredients belong to classical Herglotz--Clark,
Herglotz--Pick boundary-distribution, and Cauchy-preserver theory
\cite{clark-1972,pitman-williams-1967,williams-1969,letac-1977,
poltoratski-1996,aizenman-warzel-2015}.  Our focus is their
incidence-geometric synthesis: the
Haar pushforward, its independent planar proof, converse mechanisms, and
sharp finite and infinite boundaries.  We make no claim that the classical
Cauchy or boundary-distribution inputs are new.

The paper is organized as follows.
\cref{sec:incidence-object} introduces the radial--normal incidence geometry,
while \cref{sec:planar-engine,sec:two-frame-proof} develop the planar analytic
and random-plane foundation.  \cref{sec:finite-law-deductions} closes the
planar proof through the Beta--Cauchy equivalence and then records the
Pillai--Meng route.  The next three sections ask, successively, how reversible
and how extensible the finite theorem is.  \cref{sec:sharp-converses} shows why
the Beta marginal alone is too weak for a structural converse and identifies
stronger conditional forms that do yield rigidity.  \cref{sec:measure-arrangements}
passes from finite sums to probability and signed measures, and
\cref{sec:countable-boundary} determines when the absolute integrability needed
for that extension actually holds.  Finally, \cref{sec:equivalent-forms} gives
equivalent probabilistic and geometric forms.

\section{The incidence object}
\label{sec:incidence-object}

\subsection{Arbitrary central arrangements}

Fix an ambient dimension \(p\geq2\).  Let
\(a_1,\ldots,a_k\in\R^p\setminus\{0\}\) be any finite collection.  No
spanning, independence, or general-position hypothesis is imposed.  Let
\(w_j\geq0\), \(\sum_jw_j=1\), and discard indices with \(w_j=0\).

Define the exceptional set
\begin{equation}
  \mathcal H
  :=\bigcup_{j:w_j>0}
  \{u\in\sphere^{p-1}:a_j^\top u=0\}.
  \label{eq:exceptional-hyperplanes}
\end{equation}
It has spherical Haar measure zero.  On
\(\sphere^{p-1}\setminus\mathcal H\), define
\begin{align}
  \Phi_{a,w}(u)
  &:=\prod_{j:w_j>0}|a_j^\top u|^{w_j},
  \label{eq:master-function}\\
  g_{a,w}(u)
  &:=\nabla\log\Phi_{a,w}(u)
    =\sum_{j:w_j>0}w_j\frac{a_j}{a_j^\top u},
  \label{eq:log-gradient}\\
  N_{a,w}(u)
  &:=\frac{g_{a,w}(u)}{\|g_{a,w}(u)\|},
  \label{eq:unit-normal}\\
  H_{a,w}(u)
  &:=\frac{1}{\|g_{a,w}(u)\|^2}.
  \label{eq:incidence-functional}
\end{align}
The function \(\Phi_{a,w}\) is homogeneous of degree one.  Euler's identity
therefore gives
\begin{equation}
  u^\top g_{a,w}(u)=\sum_jw_j=1.
  \label{eq:euler-pairing}
\end{equation}
In particular, \(g_{a,w}(u)\neq0\), and \(N_{a,w}(u)\) is a well-defined
unit vector.  If \(\Theta_{a,w}(u)\) is the angle between \(u\) and
\(N_{a,w}(u)\), then
\begin{equation}
  \cos\Theta_{a,w}(u)
  =u^\top N_{a,w}(u)
  =\frac{1}{\|g_{a,w}(u)\|}>0,
  \qquad
  \cos^2\Theta_{a,w}(u)=H_{a,w}(u).
  \label{eq:angle-is-H}
\end{equation}
Thus \(\Theta_{a,w}(u)\in[0,\pi/2)\).  Cauchy--Schwarz applied to
\cref{eq:euler-pairing} also gives \(0<H_{a,w}(u)\leq1\).

\subsection{Tangent-hyperplane meaning}

For fixed \(u\notin\mathcal H\), consider the level hypersurface
\begin{equation}
  \mathcal M_u
  :=\{x:\Phi_{a,w}(x)=\Phi_{a,w}(u)\}
  \label{eq:level-surface}
\end{equation}
inside the chamber containing \(u\).  Its normal at \(u\) is
\(g_{a,w}(u)\).  By \cref{eq:euler-pairing}, its affine tangent hyperplane is
\begin{equation}
  \mathcal T_u
  =\{x:g_{a,w}(u)^\top x=1\}.
  \label{eq:tangent-hyperplane}
\end{equation}
Consequently,
\begin{equation}
  H_{a,w}(u)
  =\dist(0,\mathcal T_u)^2.
  \label{eq:tangent-distance}
\end{equation}
Thus the incidence variable is simultaneously the squared cosine in
\cref{eq:angle-is-H} and the squared origin-to-tangent-plane distance in
\cref{eq:tangent-distance}.

\subsection{Universal law}

\begin{theorem}[Universal Beta incidence-angle law]
\label{thm:universal-beta}
Let \(p\geq2\), let \(U\sim\Unif(\sphere^{p-1})\), and let the finite
collection \(a_1,\ldots,a_k\) and the simplex weights satisfy the conditions
above.  Then
\begin{equation}
  H_{a,w}(U)
  =\bigl\{U^\top N_{a,w}(U)\bigr\}^2
  \sim\Beta\!\left(\frac12,\frac{p-1}{2}\right).
  \label{eq:universal-beta}
\end{equation}
The distribution is independent of \(k\), of the locations and linear
relations among the \(a_j\)'s, and of the simplex weights.
\end{theorem}

If \(U_0,V_0\) are independent Haar directions on \(\sphere^{p-1}\), then
\begin{equation}
  (U_0^\top V_0)^2
  \sim\Beta\!\left(\frac12,\frac{p-1}{2}\right).
  \label{eq:independent-incidence}
\end{equation}
Hence \cref{thm:universal-beta} says
\begin{equation}
  U^\top N_{a,w}(U)
  \overset{d}{=}|U_0^\top V_0|.
  \label{eq:dependent-mimics-independent}
\end{equation}
The equality concerns only this one scalar observable.  The direction
\(N_{a,w}(U)\) is a deterministic function of \(U\), need not be Haar, and is
not independent of \(U\).

\section{The planar analytic engine}
\label{sec:planar-engine}

The independent planar proof begins with a one-dimensional boundary theorem.
Write
\(\mathbb T_\pi=\R/(\pi\mathbb Z)\), equipped with normalized Lebesgue
measure \(m_\pi\).

\subsection{Historical placement}

The finite rational ancestor is Boole's 1857 level-set identity for sums of
simple fractions with positive residues \cite{boole-1857}.  A probabilistic
line developed through the common-phase Cauchy preservers of Pitman--Williams
and Williams \cite{pitman-williams-1967,williams-1969}; its positive weighted
shifted-tangent identity was later used as Lemma~3.1 by Pillai--Meng
\cite{pillai-meng-2016}.  In a parallel disk-theoretic line, Nordgren proved
that the boundary map of an inner function fixing the origin preserves circle
Haar measure \cite{nordgren-1968}.  Herglotz--Clark representation turns that
statement into a Cauchy boundary law, and Letac placed such boundary functions
inside the theory of Cauchy-preserving maps
\cite{clark-1972,letac-1977,cima-matheson-ross-2006}.

Aizenman--Warzel made the half-plane connection especially explicit: their
Herglotz--Pick framework yields Cauchy boundary laws for positive singular
spectral measures and extends Boole's identity beyond finite point measures
\cite{aizenman-warzel-2015}.  For signed and complex measures,
Poltoratski's weak-* boundary-distribution theorem recovers the singular
total-variation measure from large values of the Cauchy transform
\cite{poltoratski-1996}; see also the signed/complex formulation in
\cite[Theorem~8.3]{cima-matheson-ross-2006}.  Thus the distributional
principles below are classical.  Their role here is to provide the precise
circle formulation needed by the incidence problem.  The secant condition
both identifies the boundary value with an ordinary tangent integral that is
absolutely convergent and forces the representing measure to be singular.

\subsection{A circle formulation}

\begin{lemma}[Secant integrability forces singularity]
\label{lem:secant-singularity}
Let \(\nu\) be a finite Borel measure on \(\mathbb T_\pi\).  If
\begin{equation}
  J_\nu(\phi)
  :=
  \int_{\mathbb T_\pi}
  \frac{1}{|\cos(\phi-\alpha)|}\,\nu(\dd\alpha)
  <\infty
  \quad\text{for \(m_\pi\)-almost every \(\phi\),}
  \label{eq:secant-integrability}
\end{equation}
then \(\nu\) is singular with respect to \(m_\pi\).
\end{lemma}

\begin{proof}
Write the absolutely continuous part of \(\nu\), in a local real
representative of the circle, as \(f(\alpha)\,\dd\alpha\).  If it is
nonzero, the set of Lebesgue points \(x\) with \(f(x)>0\) has positive
measure.  At each such point,
\[
  \int_{d_\pi(\alpha,x)<\varepsilon}
  \frac{f(\alpha)}{d_\pi(\alpha,x)}\,\dd\alpha
  =\infty,
\]
where \(d_\pi\) is circular distance.  Indeed, the mass of the
radius-\(t\) interval is asymptotic to \(2f(x)t\), and layer cake against
\(t^{-2}\dd t\) diverges.  For
\(\phi=x+\pi/2\) modulo \(\pi\),
\[
  |\cos(\phi-\alpha)|
  =|\sin(x-\alpha)|
  \leq d_\pi(\alpha,x)
\]
locally.  Thus \(J_\nu(\phi)=\infty\) on a positive-measure translate,
contrary to \cref{eq:secant-integrability}.
\end{proof}

\begin{theorem}[Planar Herglotz--Cauchy principle]
\label{thm:planar-cauchy}
Let \(\nu\) be a Borel probability measure on \(\mathbb T_\pi\) satisfying
\cref{eq:secant-integrability}.  Then the ordinary integral
\begin{equation}
  F_\nu(\phi)
  :=
  \int_{\mathbb T_\pi}\tan(\phi-\alpha)\,\nu(\dd\alpha)
  \label{eq:planar-tangent-transform}
\end{equation}
is absolutely convergent for \(m_\pi\)-almost every \(\phi\), and
\begin{equation}
  F_\nu(\Phi)\sim\Cauchy(0,1),
  \qquad \Phi\sim m_\pi.
  \label{eq:planar-cauchy-law}
\end{equation}
The representing measure \(\nu\) is necessarily singular.
\end{theorem}

\begin{proof}
Absolute convergence follows from
\(|\tan x|\leq|\cos x|^{-1}\), and singularity is the conclusion of
\cref{lem:secant-singularity}.  Define, for \(z\in\D\),
\begin{equation}
  \mathcal F_\nu(z)
  :=
  \int_{\mathbb T_\pi}
  i\frac{1-e^{-2i\alpha}z}{1+e^{-2i\alpha}z}
  \,\nu(\dd\alpha).
  \label{eq:planar-herglotz-integral}
\end{equation}
This analytic function maps \(\D\) into the upper half-plane and satisfies
\(\mathcal F_\nu(0)=i\).  If \cref{eq:secant-integrability} holds at
\(\phi\), then, for \(1/2\leq\rho<1\),
\[
  \left|
    i\frac{1-\rho e^{2i(\phi-\alpha)}}
            {1+\rho e^{2i(\phi-\alpha)}}
  \right|
  \leq
  \frac{\sqrt2}{|\cos(\phi-\alpha)|}.
\]
Dominated convergence as \(\rho\uparrow1\) therefore identifies the radial
boundary value
\begin{equation}
  \mathcal F_\nu^*(e^{2i\phi})=F_\nu(\phi)
  \quad\text{for \(m_\pi\)-almost every \(\phi\).}
  \label{eq:planar-herglotz-boundary}
\end{equation}

The Cayley transform
\begin{equation}
  B_\nu(z)
  :=
  \frac{1+i\mathcal F_\nu(z)}
       {1-i\mathcal F_\nu(z)}
  \label{eq:planar-cayley-map}
\end{equation}
is a Schur function, \(B_\nu(0)=0\), and
\cref{eq:planar-herglotz-boundary} gives unimodular boundary values almost
everywhere.  Let \(m_{\mathbb T}\) be Haar probability on the unit circle,
let \(B_\nu^*\) denote the radial boundary function, and set
\(\lambda=(B_\nu^*)_\#m_{\mathbb T}\).  For every integer \(n\geq1\),
the mean-value identity followed by dominated radial convergence gives
\[
  \int_{\mathbb T}\zeta^n\,\lambda(\dd\zeta)
  =
  \int_{\mathbb T}\{B_\nu^*(\zeta)\}^n
       \,m_{\mathbb T}(\dd\zeta)
  =B_\nu(0)^n=0.
\]
The negative moments vanish by conjugation.  Fourier--Stieltjes uniqueness
therefore gives \(\lambda=m_{\mathbb T}\); this is the standard
measure-preserving property of an inner function fixing the origin
\cite{nordgren-1968}.

Finally, \(2\Phi\) is uniform modulo \(2\pi\), and the inverse Cayley
transform sends circle Haar measure to the standard Cauchy law.  Combining
this fact with \cref{eq:planar-herglotz-boundary} proves
\cref{eq:planar-cauchy-law}.
\end{proof}

\begin{corollary}[Weighted shifted-tangent identity]
\label{cor:weighted-tangent}
For arbitrary \(\alpha_1,\ldots,\alpha_k\in\mathbb T_\pi\), simplex
weights \(w_1,\ldots,w_k\), and \(\Phi\sim m_\pi\),
\begin{equation}
  \sum_{j=1}^k w_j\tan(\Phi-\alpha_j)
  \sim\Cauchy(0,1).
  \label{eq:weighted-tangent-preview}
\end{equation}
\end{corollary}

\begin{proof}
Apply \cref{thm:planar-cauchy} to
\(\nu=\sum_jw_j\delta_{\alpha_j}\).  Its secant integral is finite away
from finitely many phases.
\end{proof}

\begin{lemma}[Pole-tail formula for signed tangent sums]
\label{lem:signed-tangent-tail}
Let \(\alpha_1,\ldots,\alpha_m\) be distinct modulo \(\pi\), let
\(c_1,\ldots,c_m\) be nonzero real numbers, and let
\[
  F(\phi):=\sum_{\ell=1}^m c_\ell\tan(\phi-\alpha_\ell).
\]
For \(\Phi\sim m_\pi\),
\begin{equation}
  \lim_{x\to\infty}
  x\,\Prb\{|F(\Phi)|>x\}
  =
  \frac2\pi\sum_{\ell=1}^m|c_\ell|.
  \label{eq:signed-tangent-tail}
\end{equation}
Consequently, if \(\sum_\ell c_\ell=1\) and \(F(\Phi)\) is standard
Cauchy, then every \(c_\ell\) is positive.
\end{lemma}

\begin{proof}
The poles \(\beta_\ell=\alpha_\ell+\pi/2\) are distinct.  Choose disjoint
circle neighborhoods of them.  As \(h\to0\),
\[
  F(\beta_\ell+h)=-\frac{c_\ell}{h}+O(1),
\]
while \(F\) is bounded outside those neighborhoods.  A squeeze at each pole
therefore shows that the set where \(|F|>x\) has local length
\(2|c_\ell|x^{-1}+o(x^{-1})\).  Summing the disjoint contributions and
dividing by the circle length \(\pi\) proves
\cref{eq:signed-tangent-tail}.  A standard Cauchy variable has tail
constant \(2/\pi\).  Hence
\(\sum_\ell|c_\ell|=1=|\sum_\ell c_\ell|\), and equality in the triangle
inequality forces all coefficients to have the same positive sign.
\end{proof}

The distinctness assumption is essential because coincident poles reveal
only the combined residue.  For example,
\(2\tan\Phi-\tan\Phi=\tan\Phi\) is standard Cauchy although one displayed
coefficient is negative.  This is why projectively coincident directions
are combined before \cref{lem:signed-tangent-tail} is used below.

\begin{theorem}[Signed one-plane rigidity]
\label{thm:signed-clark-rigidity}
Let \(\nu\) be a finite real signed measure on \(\mathbb T_\pi\), singular
with respect to \(m_\pi\), and suppose that \(\nu(\mathbb T_\pi)=1\).
Define
\[
  \mathcal F_\nu(z)
  :=
  \int_{\mathbb T_\pi}
  i\frac{1-e^{-2i\alpha}z}{1+e^{-2i\alpha}z}
  \,\nu(\dd\alpha),
\]
and let \(\mathcal F_\nu^*\) be its nontangential boundary value, which is
finite and real almost everywhere.  If
\(S(\alpha)=\alpha+\pi/2\) modulo \(\pi\), then, weak-* as finite measures,
\begin{equation}
  \pi x\,\boldsymbol 1_{\{|\mathcal F_\nu^*|>x\}}m_\pi
  \longrightarrow 2S_\#|\nu|.
  \label{eq:signed-clark-local-tail}
\end{equation}
In particular,
\begin{equation}
  \lim_{x\to\infty}
  x\,m_\pi\{|\mathcal F_\nu^*(e^{2i\phi})|>x\}
  =
  \frac2\pi\|\nu\|_{\mathrm{TV}}.
  \label{eq:signed-clark-tail}
\end{equation}
Consequently,
\[
  \mathcal F_\nu^*(e^{2i\Phi})\sim\Cauchy(0,1)
  \quad\Longleftrightarrow\quad
  \nu\ \text{is a positive probability measure}.
\]
\end{theorem}

\begin{proof}
Push \(\nu\) to a signed measure \(\tau\) on the unit circle under
\(\alpha\mapsto-e^{2i\alpha}\), and write
\[
  K_\tau(z):=\int_{\mathbb T}\frac{1}{1-\overline\zeta z}
  \,\tau(\dd\zeta).
\]
Then \(\mathcal F_\nu=i(2K_\tau-1)\).  The boundary-distribution theorem for
Cauchy integrals reconstructs the singular total-variation measure weak-*
from the two-sided large-value sets: for every finite complex measure
\(\tau\),
\[
  \pi\lambda\,\boldsymbol 1_{\{|K_\tau^*|>\lambda\}}m_{\mathbb T}
  \longrightarrow |\tau_s|.
\]
This is Poltoratski's signed/complex theorem \cite{poltoratski-1996}; see
also \cite[Theorem~8.3]{cima-matheson-ross-2006}.  Standard Cauchy-transform
boundary theory also supplies finite nontangential boundary values almost
everywhere.  Since \(\tau\) is real, singular, and has mass one,
\[
  2\operatorname{Re}K_\tau-1=P[\tau]
\]
where \(P[\tau]\) is the Poisson integral.  Fatou's theorem therefore gives
\(\operatorname{Re}(2K_\tau^*-1)=0\) almost everywhere.  Consequently,
\[
  |K_\tau^*|=\frac12\sqrt{1+|\mathcal F_\nu^*|^2}.
\]
Taking \(\lambda=\sqrt{1+x^2}/2\), noting that \(x/\lambda\to2\), and
transporting circle Haar measure through the phase parametrization proves
\cref{eq:signed-clark-local-tail}; the pole corresponding to \(\alpha\)
occurs at \(S(\alpha)\).
Taking total masses in \cref{eq:signed-clark-local-tail} proves
\cref{eq:signed-clark-tail}.  If the boundary variable is
standard Cauchy, its two-sided tail constant is \(2/\pi\), so
\(\|\nu\|_{\mathrm{TV}}=1=\nu(\mathbb T_\pi)\).  The Jordan decomposition
then forces the negative part of \(\nu\) to vanish.  Conversely, positivity
reduces the claim to the inner-function argument in
\cref{thm:planar-cauchy}; no secant hypothesis is needed for the
nontangential boundary statement.
\end{proof}

\begin{remark}
\cref{thm:signed-clark-rigidity} concerns nontangential Herglotz boundary
values.  To identify those values with the ordinary integral
\(\int\tan(\phi-\alpha)\,\nu(\dd\alpha)\), one must separately assume
absolute secant integrability for the total variation measure.  The finite
pole formula in \cref{lem:signed-tangent-tail} is the elementary atomic
case of the same tail principle.
\end{remark}

\begin{lemma}[Ordinary tangent integrals are Herglotz boundary values]
\label{lem:signed-ordinary-boundary}
Let \(\nu\) be a finite real signed measure on \(\mathbb T_\pi\) and suppose
that
\[
  J_{|\nu|}(\phi)
  =\int_{\mathbb T_\pi}
    \frac{1}{|\cos(\phi-\alpha)|}\,|\nu|(\dd\alpha)
  <\infty
\]
for almost every \(\phi\).  Then the ordinary integral
\[
  \int_{\mathbb T_\pi}\tan(\phi-\alpha)\,\nu(\dd\alpha)
\]
is absolutely convergent and agrees almost everywhere with the
nontangential boundary value of \(\mathcal F_\nu\).
\end{lemma}

\begin{proof}
The domination estimate used in the proof of \cref{thm:planar-cauchy}
applies to \(|\nu|\) after normalization and identifies the radial limit by
dominated convergence.  A Cauchy transform of a finite signed measure is a
difference of two positive Cauchy transforms and belongs to \(H^q\) for
every \(0<q<1\).  Its radial and nontangential boundary values therefore
exist and agree almost everywhere \cite{duren-1970,cima-matheson-ross-2006}.
\end{proof}

\begin{theorem}[Oriented analytic converse]
\label{thm:analytic-clark-converse}
Let \(h:\mathbb T_\pi\to\R\) be measurable and finite almost everywhere,
and put
\begin{equation}
  b_h(e^{2i\phi})
  :=
  \frac{1+ih(\phi)}{1-ih(\phi)}.
  \label{eq:boundary-cayley}
\end{equation}
Assume that \(b_h\) is the radial boundary function of a Schur function
\(B\) on \(\D\).  Then \(B\) is inner, and
\[
  h(\Phi)\sim\Cauchy(0,1)
  \quad\Longleftrightarrow\quad
  B(0)=0.
\]
When these equivalent conditions hold, \(i(1-B)/(1+B)\) has a unique Herglotz
representation by a singular probability measure.  Conversely, every
singular probability measure gives an inner \(B\) fixing the origin and a
Cauchy boundary function in this way.
\end{theorem}

\begin{proof}
Because \(h\) is real, \(|b_h|=1\) almost everywhere, so the assumed Schur
function is inner.  If \(h(\Phi)\) is standard Cauchy, then
\((b_h)_\#m_{\mathbb T}=m_{\mathbb T}\), whence
\[
  B(0)=\int_{\mathbb T}b_h\,\dd m_{\mathbb T}=0.
\]
Conversely, if \(B(0)=0\), the moment argument in the proof of
\cref{thm:planar-cauchy} shows that its boundary map preserves Haar
measure, and inverse Cayley gives the Cauchy law.  The remaining assertions
are the normalized Herglotz representation and the Clark correspondence
\cite{clark-1972,cima-matheson-ross-2006}.
\end{proof}

\begin{remark}[Why Cauchy preservation alone is not a Clark converse]
\label{rem:measurable-not-inner}
Without the Schur boundary hypothesis in
\cref{thm:analytic-clark-converse}, the conclusion is false.  For example,
\[
  h(\phi)=-\tan\phi,
  \qquad
  b_h(e^{2i\phi})=e^{-2i\phi}.
\]
The first variable is standard Cauchy and the second map preserves Haar
measure, but it is anti-analytic and is not the boundary function of any
member of \(H^\infty(\D)\).  Thus the Cauchy law identifies measurable
circle-measure preservation, not an analytic orientation.  Even in the
inner case, Herglotz theory gives a nontangential boundary value; the
ordinary integral in \cref{eq:planar-tangent-transform} requires the
separate secant condition \cref{eq:secant-integrability}.
\end{remark}

\section{The random-two-frame proof}
\label{sec:two-frame-proof}

The argument combines spherical Haar geometry, the weighted tangent identity
in \cref{eq:weighted-tangent-preview}, and a Mellin identification step.

\subsection{Uniform phase on a random plane}

Let \(U\) be Haar on \(\sphere^{p-1}\).  Conditional on \(U\), choose \(V\)
Haar-uniform on
\[
  \{v\in\sphere^{p-1}:v^\top U=0\}.
\]
Then \((U,V)\) is a Haar orthonormal two-frame.  Equivalently, it is the first
two columns of a Haar matrix in \(O(p)\).

\begin{lemma}[Uniform orientation within the random plane]
\label{lem:uniform-frame-phase}
Let \(P=\operatorname{span}\{U,V\}\), equipped with the orientation induced
by the ordered frame.  Conditional on the oriented plane \(P\), the frame's
in-plane rotation is Haar-uniform on \(SO(2)\).
\end{lemma}

\begin{proof}
For every deterministic planar rotation \(R_\delta\), right invariance of
Haar measure on \(O(p)\) gives
\[
 [\,U\;V\,]R_\delta\overset{d}{=}[\,U\;V\,].
\]
Right rotation leaves the oriented plane fixed and shifts the in-plane angle
by \(\delta\).  Disintegrate the frame law over the oriented Grassmannian and
first take \(\delta\) in a countable dense subgroup of \(SO(2)\).  Outside
one common null set of planes, every corresponding conditional law is
invariant under all those shifts.  Weak continuity of translation then
extends the invariance to every \(\delta\in SO(2)\).  Uniqueness of Haar
probability on \(SO(2)\) makes the conditional law uniform.
\end{proof}

For an equivalent Gaussian construction of the same invariance, take a
\(p\times2\) matrix \(G\) of independent standard normals and set
\(Q=G(G^\top G)^{-1/2}\).  Then \(Q=[\,U\;V\,]\), and isotropy gives
\(GR_\delta\overset{d}{=}G\).  The equivariance
\(Q(GR_\delta)=Q(G)R_\delta\) proves the same right-rotation invariance.

\subsection{The Cauchy tangent projection}

Decompose the logarithmic gradient into radial and tangent components:
\begin{equation}
  g_{a,w}(U)=U+T(U),
  \qquad
  T(U):=g_{a,w}(U)-U\in U^\perp.
  \label{eq:tangent-decomposition}
\end{equation}
The equality of the radial component to \(U\) follows from
\cref{eq:euler-pairing}.  Define the scalar random tangent projection
\begin{equation}
  Z:=T(U)^\top V.
  \label{eq:random-projection}
\end{equation}
Since \(U^\top V=0\),
\begin{equation}
  Z
  =g_{a,w}(U)^\top V
  =\sum_{j=1}^k w_j
    \frac{a_j^\top V}{a_j^\top U}.
  \label{eq:Z-ratio-sum}
\end{equation}

Condition on the oriented plane \(P\), not on the full pair \((U,V)\).  Fix
an oriented orthonormal basis \((e_1,e_2)\) of \(P\).  By
\cref{lem:uniform-frame-phase}, for one common uniform angle \(\phi\),
\begin{equation}
  U=e_1\cos\phi+e_2\sin\phi,
  \qquad
  V=-e_1\sin\phi+e_2\cos\phi.
  \label{eq:frame-parametrization}
\end{equation}
For almost every \(P\), the projection of every active \(a_j\) onto \(P\) is
nonzero.  Write it as
\[
  \operatorname{proj}_P(a_j)
  =r_j(e_1\cos\alpha_j+e_2\sin\alpha_j),
  \qquad r_j>0.
\]
Then, simultaneously for all \(j\),
\begin{equation}
  \frac{a_j^\top V}{a_j^\top U}
  =-\tan(\phi-\alpha_j).
  \label{eq:common-phase-ratios}
\end{equation}
The same \(\phi\) occurs in every term; the arrangement only determines the
fixed shifts.  The uniform angle modulo \(2\pi\) is uniform modulo \(\pi\),
and the minus sign in \cref{eq:common-phase-ratios} is immaterial by symmetry
of the Cauchy law.  Applying \cref{eq:weighted-tangent-preview} conditionally
on \(P\) gives
\begin{equation}
  \Law(Z\mid P)=\Cauchy(0,1)
  \quad\text{for almost every }P,
  \qquad
  Z\sim\Cauchy(0,1).
  \label{eq:Z-cauchy}
\end{equation}
No comparison of angular coordinates across different random planes is
needed: the conditional law is the same for almost every plane.

\section{The Beta--Cauchy equivalence and the Pillai--Meng route}
\label{sec:finite-law-deductions}

The planar analysis above has two roles: it supplies the Cauchy projection in
the two-frame proof, and the same weighted-tangent mechanism appears in the
original Pillai--Meng proof.  We first close the two-frame argument with the
exact Beta--Cauchy tangent-projection equivalence and then record the
alternative Gaussian derivation through the Pillai--Meng theorem.

\subsection{The Beta--Cauchy tangent-projection equivalence}
\label{sec:tangent-projection-converse}

The following equivalence does not require an arrangement structure; it
holds for every measurable tangent field.

\begin{theorem}[Beta--Cauchy tangent-projection equivalence]
\label{thm:tangent-projection-equivalence}
Let \(p\geq2\), let \(U\sim\Unif(\sphere^{p-1})\), and let
\(T:\sphere^{p-1}\to\R^p\) be a measurable vector field satisfying
\[
  u^\top T(u)=0
\]
for Haar-almost every \(u\).  Conditional on \(U\), let \(V\) be
Haar-uniform on the unit sphere in \(U^\perp\), and define
\begin{equation}
  R:=\|T(U)\|,
  \qquad
  Z_T:=T(U)^\top V,
  \qquad
  H_T:=\frac{1}{1+R^2}.
  \label{eq:general-tangent-variables}
\end{equation}
Then
\begin{equation}
  H_T\sim\Beta\!\left(\frac12,\frac{p-1}{2}\right)
  \quad\Longleftrightarrow\quad
  Z_T\sim\Cauchy(0,1).
  \label{eq:beta-cauchy-equivalence}
\end{equation}
Equivalently, the universal Beta law characterizes the distribution of the
tangent norm through one conditionally Haar tangent projection.
\end{theorem}

\begin{proof}
On either side of \cref{eq:beta-cauchy-equivalence}, \(R>0\) almost surely.
For the Beta law this follows because the distribution has no mass at one.
For the Cauchy law it follows from
\(\{R=0\}\subseteq\{Z_T=0\}\).

On the event \(R>0\), define \(Q:=Z_T/R\).  Given \(U\), isotropy of \(V\)
then gives the almost-sure factorization
\begin{equation}
  Z_T=RQ,
  \label{eq:general-RQ-factorization}
\end{equation}
where \(Q\) has the law of the first coordinate of a Haar point on
\(\sphere^{p-2}\).  Its conditional law does not depend on \(U\), so \(Q\)
is independent of \(R\).  When \(p=2\), \(Q\) is Rademacher.  For every
\(p\geq2\),
\begin{equation}
  \E|Q|^s
  =
  \frac{
    \Gamma((s+1)/2)\Gamma((p-1)/2)
  }{
    \sqrt\pi\,\Gamma((p-1+s)/2)
  },
  \qquad s>-1.
  \label{eq:general-Q-mellin}
\end{equation}

Suppose first that \(H_T\) has the asserted Beta law.  The transformation in
\cref{eq:general-tangent-variables} gives
\[
  R^2\sim\BetaPrime\!\left(\frac{p-1}{2},\frac12\right)
\]
and hence
\begin{equation}
  \E R^s
  =
  \frac{
    \Gamma((p-1+s)/2)\Gamma((1-s)/2)
  }{
    \Gamma((p-1)/2)\sqrt\pi
  },
  \qquad -(p-1)<s<1.
  \label{eq:general-R-mellin}
\end{equation}
Combining
\cref{eq:general-RQ-factorization,eq:general-Q-mellin,eq:general-R-mellin}
gives
\[
  \E|Z_T|^s
  =
  \frac{
    \Gamma((1-s)/2)\Gamma((1+s)/2)
  }{\pi}
  =
  \frac{1}{\cos(\pi s/2)},
  \qquad -1<s<1.
\]
This is the Mellin transform of the absolute value of a standard Cauchy
variable.  It identifies the law because it is the moment-generating
function of \(\log|Z_T|\) on an open interval containing zero.  The
conditional distribution of \(Z_T\) given \(U\) is symmetric, so \(Z_T\)
itself is standard Cauchy.

Conversely, suppose \(Z_T\sim\Cauchy(0,1)\).  Tonelli's theorem and the
independence in \cref{eq:general-RQ-factorization} give, for
\(-1<s<1\),
\[
  \E R^s
  =
  \frac{\E|Z_T|^s}{\E|Q|^s}.
\]
Using the Cauchy Mellin transform and \cref{eq:general-Q-mellin} recovers
\cref{eq:general-R-mellin}.  The moment-generating function of \(\log R\)
therefore identifies
\[
  R^2\sim\BetaPrime\!\left(\frac{p-1}{2},\frac12\right).
\]
The transformation \(H_T=(1+R^2)^{-1}\) proves the required Beta law.
\end{proof}

For the arrangement field in \cref{eq:tangent-decomposition},
\cref{eq:Z-cauchy,thm:tangent-projection-equivalence} now give
\[
  R^2\sim\BetaPrime\!\left(\frac{p-1}{2},\frac12\right),
  \qquad
  H_{a,w}(U)\sim\Beta\!\left(\frac12,\frac{p-1}{2}\right),
\]
which completes the planar proof of \cref{thm:universal-beta}.

\subsection{The Pillai--Meng Gaussian route}
\label{sec:pm-input}

A second route passes through the positive-semidefinite Gaussian ratio
identity of Pillai and Meng \cite{pillai-meng-2016}.  Their proof also invokes
the weighted shifted-tangent identity in \cref{eq:weighted-tangent-preview}:
an angular change of variables extracts one common uniform phase conditional
on the remaining phase differences.  Here we instead take their theorem as
the input and derive the incidence law from it.

Let \(X,Y\in\R^m\) be independent with
\[
  X,Y\sim N_m(0,\Sigma),
\]
where \(\Sigma\) is positive semidefinite and
\(\Sigma_{jj}>0\) for every \(j\).  For deterministic simplex weights,
Pillai and Meng proved
\begin{equation}
  \sum_{j=1}^m w_j\frac{X_j}{Y_j}
  \sim\Cauchy(0,1).
  \label{eq:pm-cauchy}
\end{equation}
Conditioning the left side on \(Y\) gives a centered normal variable with
random variance
\begin{equation}
  L_{\Sigma,w}(Y)
  :=
  \left(\frac{w}{Y}\right)^{\!\top}
  \Sigma
  \left(\frac{w}{Y}\right),
  \qquad
  \frac{w}{Y}:=\left(\frac{w_j}{Y_j}\right)_{j=1}^m.
  \label{eq:levy-functional}
\end{equation}
Uniqueness of centered Gaussian variance mixtures yields
\begin{equation}
  L_{\Sigma,w}(Y)^{-1}\sim\chi_1^2.
  \label{eq:pm-levy}
\end{equation}

Apply this result with \(m=k\).  Let \(A=[a_1,\ldots,a_k]\), put
\(\Sigma=A^\top A\), and write a standard Gaussian vector as \(Z=RU\), where
\(R^2\sim\chi_p^2\) is independent of \(U\).  With
\(Y=A^\top Z=RA^\top U\),
\[
  L_{A^\top A,w}(A^\top Z)
  =
  \frac{\|g_{a,w}(U)\|^2}{R^2}.
\]
Thus \cref{eq:pm-levy} gives
\(R^2/\|g_{a,w}(U)\|^2\sim\chi_1^2\).  Beta--gamma Mellin
deconvolution, using the independence of \(R\) and \(U\), proves
\cref{eq:universal-beta}.  Singular Gram matrices are allowed, so repeated,
dependent, and overcomplete arrangements are included.

The Gaussian route therefore locates the finite marginal inside the
Pillai--Meng Cauchy--L\'evy identity.  The planar route additionally retains
the oriented boundary map and the conditional plane-wise phase used in the
converse and infinite-arrangement results below.

\section{Three levels of rigidity and finite scope}
\label{sec:sharp-converses}

The universal theorem is a forward implication: positive reciprocal
coefficients produce the Beta incidence marginal.  This section separates
three possible reversals.  The marginal itself determines only the tangent
norm and admits non-gradient counterexamples.  Requiring the Cauchy law on
almost every oriented two-plane recovers positivity for a fixed reciprocal
arrangement and determines the magnitude in a nonlinear zonal class, although
orientation remains ambiguous.  Requiring universality over every finite
arrangement yields the sharp coefficient-level converse: the simplex is the
maximal normalized real-weight class.

\subsection{Degenerate and boundary cases}
\label{sec:edge-cases}

The proof records the following edge cases explicitly.

\begin{itemize}
\item The theorem is stated for \(p\geq2\).  For \(p=1\), there is no random
  tangent direction and the incidence square is identically one; the displayed
  Beta parameter would be zero.
\item When \(p=2\), the tangent unit sphere is \(\sphere^0\), so \(Q\) is a
  sign and \(R=|Z|\).  Thus
  \(R^2\sim\BetaPrime(1/2,1/2)\), and the equivalence remains valid with
  this Rademacher projection factor.
\item Zero weights are removed before defining the exceptional set.  They
  contribute neither a logarithm nor a ratio.
\item Every active \(a_j\) must be nonzero.  Then
  \(\Prb(a_j^\top U=0)=0\).  For a finite arrangement, their union remains
  null.
\item A fixed nonzero \(a_j\) has zero projection onto a Haar-random
  two-plane only on a Grassmannian-null event.  Thus the planar polar
  coordinates used in \cref{eq:common-phase-ratios} are valid almost surely.
\item Linear dependence, repeated hyperplanes, collinearity, and failure to
  span \(\R^p\) cause no difficulty.  The proof never inverts the arrangement
  matrix.
\end{itemize}

\subsection{The Beta marginal does not characterize logarithmic gradients}

The converse in \cref{thm:tangent-projection-equivalence} concerns only the
tangent-norm distribution.  It does not recover a logarithmic-gradient,
arrangement, or plane-conditional common-phase structure.

\begin{example}[A non-gradient swirl with the universal Beta marginal]
\label{ex:swirl-counterexample}
Let \(p=3\), write \(u=(x,y,z)\in\sphere^2\), and let \(e_1=(1,0,0)\).
Away from the great circle \(\{x=0\}\), define
\begin{equation}
  T_{\mathrm{sw}}(u)
  :=
  \frac{e_1\times u}{|e_1^\top u|}
  =
  \frac{(0,-z,y)}{|x|},
  \qquad
  G_{\mathrm{sw}}(u):=u+T_{\mathrm{sw}}(u).
  \label{eq:swirl-field}
\end{equation}
Extend \(T_{\mathrm{sw}}\) by zero on \(\{x=0\}\); this makes it a
measurable tangent field on all of \(\sphere^2\) without changing any
distribution below.
The field is tangent because \(u^\top T_{\mathrm{sw}}(u)=0\), and
\[
  \|T_{\mathrm{sw}}(u)\|^2
  =\frac{1-x^2}{x^2}.
\]
Consequently, if
\(N_{\mathrm{sw}}=G_{\mathrm{sw}}/\|G_{\mathrm{sw}}\|\), then
\begin{equation}
  \{u^\top N_{\mathrm{sw}}(u)\}^2
  =
  \frac{1}{1+\|T_{\mathrm{sw}}(u)\|^2}
  =x^2.
  \label{eq:swirl-incidence}
\end{equation}
For \(U\sim\Unif(\sphere^2)\), this is exactly
\(\Beta(1/2,1)\), the universal incidence law in dimension three.

Nevertheless, \(T_{\mathrm{sw}}\) is not even locally a spherical gradient.
On the hemisphere \(x>0\), use the chart
\[
  (y,z)\longmapsto
  \bigl(\sqrt{1-y^2-z^2},y,z\bigr)
\]
and put \(x=\sqrt{1-y^2-z^2}\).  The metric-dual one-form of
\(T_{\mathrm{sw}}\) is
\begin{equation}
  \omega_{\mathrm{sw}}
  =
  -\frac{z}{x}\,\dd y+\frac{y}{x}\,\dd z.
  \label{eq:swirl-one-form}
\end{equation}
A direct calculation gives
\begin{equation}
  \dd\omega_{\mathrm{sw}}
  =
  \frac{1+x^2}{x^3}\,\dd y\wedge\dd z,
  \label{eq:swirl-not-closed}
\end{equation}
which is nowhere zero on this hemisphere.  Thus
\(\omega_{\mathrm{sw}}\) is not locally exact.  In particular,
\(G_{\mathrm{sw}}\) cannot be the normalized logarithmic gradient of a
differentiable homogeneous function, and it cannot be an arrangement field
of the form \cref{eq:log-gradient}.  Indeed, the tangential part of an
ambient gradient restricted to the sphere is the intrinsic spherical
gradient of the restricted potential, whose metric-dual one-form is locally
exact.  Here a hypothetical primitive would automatically be smooth because
its differential is the smooth form \(\omega_{\mathrm{sw}}\), contradicting
\cref{eq:swirl-not-closed}.

The plane-conditional common-phase property also fails.  Let \((U,V)\) be a
Haar oriented two-frame and let \(n=U\times V\), so \(V=n\times U\).
Conditional on the oriented plane, \(n\) is fixed and
\[
  T_{\mathrm{sw}}(U)^\top V
  =
  \frac{(e_1\times U)^\top(n\times U)}
       {|e_1^\top U|}
  =
  \frac{e_1^\top n}{|e_1^\top U|}.
\]
Put \(c=e_1^\top n\) and \(r=\sqrt{1-c^2}\).  Under the conditional uniform
phase, \(e_1^\top U=r\cos\Phi\), and hence
\begin{equation}
  T_{\mathrm{sw}}(U)^\top V
  =
  \frac{c}{r|\cos\Phi|}.
  \label{eq:swirl-conditional-projection}
\end{equation}
For almost every oriented plane, \(0<|c|<1\).  The conditional variable in
\cref{eq:swirl-conditional-projection} has a fixed sign and absolute value
at least \(|c|/r\), so it is not standard Cauchy.  Its unconditional law is
nevertheless standard Cauchy by
\cref{thm:tangent-projection-equivalence,eq:swirl-incidence}.
\end{example}

There is nonuniqueness even inside the homogeneous-gradient class.  If
\(q\) is differentiable, nonvanishing, and homogeneous of degree \(d>0\),
write
\[
  g_q(x):=\frac1d\nabla\log|q(x)|
\]
for its Euler-normalized logarithmic gradient, and
define its radial dual
\begin{equation}
  q^\#(x):=\frac{\|x\|^{2d}}{q(x)}.
  \label{eq:radial-dual}
\end{equation}
Then \(q^\#\) is again homogeneous of degree \(d\), and
\begin{equation}
  g_{q^\#}(x)
  =\frac{2x}{\|x\|^2}-g_q(x).
  \label{eq:radial-dual-gradient}
\end{equation}
Thus, on the unit sphere, the radial dual changes the tangent component
\(T\) to \(-T\) and leaves the incidence square unchanged pointwise.  A
scalar Beta marginal cannot distinguish this pair.

\subsection{Conditional rigidity for a fixed finite arrangement}

\begin{theorem}[Fixed-arrangement conditional rigidity]
\label{thm:fixed-arrangement-rigidity}
Let \(a_1,\ldots,a_k\in\R^p\setminus\{0\}\), let
\(c_1,\ldots,c_k\in\R\) satisfy \(\sum_jc_j=1\), and combine all indices
whose vectors determine the same projective direction, omitting any groups
whose combined coefficient is zero.  Write \(b_1,\ldots,b_m\) for the
remaining combined coefficients and set
\[
  T_{a,c}(u)
  :=
  \sum_{j=1}^k c_j\frac{a_j}{a_j^\top u}-u.
\]
Then
\begin{equation}
  \Law\{T_{a,c}(U)^\top V\mid P\}
  =\Cauchy(0,1)
  \quad\text{for Haar-almost every oriented plane \(P\)}
  \label{eq:fixed-arrangement-conditional}
\end{equation}
if and only if \(b_\ell>0\) for every \(\ell\).
\end{theorem}

\begin{proof}
For a generic plane, all active projective directions have nonzero
projections and their projected phases
\(\alpha_1,\ldots,\alpha_m\) are distinct modulo \(\pi\).  Indeed, for two
nonparallel directions the collision condition is
\[
  (a^\top U)(b^\top V)-(a^\top V)(b^\top U)=0.
\]
For \(p\geq3\), this is the zero set of a nontrivial real-analytic function
on the connected Stiefel manifold and is therefore Haar-null.  Nontriviality
follows by taking the plane to contain the two directions.  For \(p=2\), the
same expression equals
\(\det[a\ b]\det[U\ V]\) and is nowhere zero.  A finite union of collision
sets remains null.

On each such plane, the common-phase calculation gives
\[
  -T_{a,c}(U_\phi)^\top V_\phi
  =
  \sum_{\ell=1}^m b_\ell\tan(\phi-\alpha_\ell).
\]
If \cref{eq:fixed-arrangement-conditional} holds,
\cref{lem:signed-tangent-tail} and \(\sum_\ell b_\ell=1\) force every
combined coefficient to be positive; the leading minus sign is harmless by
the symmetry of the Cauchy law.  Conversely, positive combined
coefficients are simplex weights, so \cref{cor:weighted-tangent} gives the
conditional Cauchy law.
\end{proof}

\begin{remark}
This is a fixed-arrangement statement: it does not require universality
over all possible choices of directions.  Individual signed coefficients
are not identifiable when their projective directions coincide, because
only their combined coefficient enters the field.
\end{remark}

\subsection{Zonal rigidity and its sign obstruction}

The plane-conditional law is stronger than the scalar Beta marginal, but it
still cannot recover an orientation because the Cauchy law is symmetric.
The next theorem gives a complete answer in a natural infinite-dimensional
zonal class.

\begin{theorem}[Zonal magnitude rigidity]
\label{thm:zonal-rigidity}
Let \(p\geq3\), fix \(a\in\sphere^{p-1}\), and let
\(\psi:(-1,0)\cup(0,1)\to\R\) be Borel and finite almost everywhere.
Define the measurable tangent field
\begin{equation}
  T_\psi(u)
  :=
  \psi(a^\top u)\{a-(a^\top u)u\}
  \label{eq:zonal-field}
\end{equation}
away from \(a^\perp\), with an arbitrary definition on that null set.
The following are equivalent.
\begin{enumerate}[label=\textup{(\roman*)}]
\item For Haar-almost every oriented two-plane \(P\),
  \[
    \Law\{T_\psi(U)^\top V\mid P\}=\Cauchy(0,1).
  \]
\item
  \[
    |\psi(t)|=\frac1{|t|}
    \quad\text{for Lebesgue-almost every }t\in(-1,1).
  \]
\end{enumerate}
If \(T_\psi\) is also antipodally odd, then necessarily
\begin{equation}
  T_\psi(u)
  =
  \varepsilon(|a^\top u|)
  \left\{\frac{a}{a^\top u}-u\right\},
  \qquad
  \varepsilon:(0,1)\to\{-1,1\}
  \label{eq:zonal-sign-family}
\end{equation}
for an arbitrary measurable sign function, and every such field satisfies
\textup{(i)}.  If \(\psi\) is continuous on each half-interval, the sign is
constant and only the reciprocal field and its radial dual remain.
\end{theorem}

\begin{proof}
For a nonexceptional oriented plane \(P\), choose a positive orthonormal
basis so that
\[
  \operatorname{proj}_P(a)=r e_1,
  \qquad 0<r<1.
\]
Here \(r<1\) almost surely because \(p\geq3\), and the aligned bases can be
chosen measurably on a countable Grassmannian atlas.  Conditional uniformity
of the phase from \cref{lem:uniform-frame-phase} is therefore preserved.
With
\[
  U_\theta=e_1\cos\theta+e_2\sin\theta,
  \qquad
  V_\theta=-e_1\sin\theta+e_2\cos\theta,
\]
where \(\theta\) is uniform modulo \(2\pi\), direct calculation gives
\begin{equation}
  Z_r(\theta)
  :=
  T_\psi(U_\theta)^\top V_\theta
  =
  -r\sin\theta\,\psi(r\cos\theta).
  \label{eq:zonal-projection}
\end{equation}
For a Haar plane,
\(r^2\sim\Beta(1,(p-2)/2)\); its density is strictly positive throughout
\((0,1)\).  The conditional law in \cref{eq:zonal-projection} depends only
on \(r\), so \textup{(i)} is equivalent to \(Z_r\) being standard Cauchy
for Lebesgue-almost every \(r\in(0,1)\).

Fix \(0<s<1\).  Since \(\cos\theta\) has the arcsine density, the absolute
\(s\)-moment at every such \(r\) is
\begin{equation}
  \E|Z_r|^s
  =
  \frac1\pi
  \int_{-r}^r
  (r^2-t^2)^{(s-1)/2}|\psi(t)|^s\,\dd t.
  \label{eq:zonal-moment-transform}
\end{equation}
For a standard Cauchy variable this moment equals
\(\sec(\pi s/2)\).  The same beta integral shows that replacing
\(|\psi(t)|^s\) in \cref{eq:zonal-moment-transform} by \(|t|^{-s}\)
also gives \(\sec(\pi s/2)\).

For \(t>0\), put
\[
  d_s(t)
  :=
  \frac{|\psi(t)|^s+|\psi(-t)|^s}{2}-t^{-s}.
\]
It follows that
\begin{equation}
  \int_0^r
  (r^2-t^2)^{\beta-1}d_s(t)\,\dd t=0
  \quad\text{for almost every }r,
  \qquad
  \beta:=\frac{1+s}{2}\in(1/2,1).
  \label{eq:zonal-abel-zero}
\end{equation}
The moment identity at any good radius \(r_0\) shows that
\[
  \int_0^{r_0}
  \bigl\{|\psi(t)|^s+|\psi(-t)|^s\bigr\}\,\dd t<\infty:
\]
indeed, the kernel in \cref{eq:zonal-moment-transform} is bounded below by
\(r_0^{s-1}\) on \((0,r_0)\).  Since \(t^{-s}\) is integrable at zero,
\(d_s\) is locally integrable.  In
\cref{eq:zonal-abel-zero}, substitute \(x=r^2\), \(y=t^2\), and
\[
  f_s(y):=\frac{d_s(\sqrt y)}{2\sqrt y}.
\]
The resulting equation is the vanishing of the
Riemann--Liouville fractional integral \(I_{0+}^{\beta}f_s\).  Applying
\(I_{0+}^{1-\beta}\) and using the semigroup identity gives the next
display.  The interchange is justified first for \(|f_s|\) by Tonelli and
then for \(f_s\) by Fubini; the inner integral is the beta integral.  Thus
\[
  \int_0^x f_s(y)\,\dd y=0
  \quad\text{for almost every }x.
\]
The left side is absolutely continuous, so it vanishes everywhere and
\(d_s=0\) almost everywhere.

Two exponents are needed to identify two unknown magnitudes.  Use
\(s=1/4\) and \(s=1/2\) on their common full-measure set.  If
\[
  A=t|\psi(t)|,\qquad C=t|\psi(-t)|,
\]
then \(A^{1/4}+C^{1/4}=2\) and
\(A^{1/2}+C^{1/2}=2\).  Writing \(x=A^{1/4}\) and \(y=C^{1/4}\), these
relations give \(x+y=2\) and \(x^2+y^2=2\), hence the double solution
\(x=y=1\).  This proves \textup{(ii)}.

Conversely, under \textup{(ii)},
\(|Z_r(\theta)|=|\tan\theta|\) almost everywhere.  The transformation
\(\theta\mapsto-\theta\) preserves \(\cos\theta\) and reverses the sign in
\cref{eq:zonal-projection}; hence \(Z_r\) is symmetric and has the
half-Cauchy absolute-value law.  It is therefore standard Cauchy.

Finally, \(T_\psi(-u)=-T_\psi(u)\) is equivalent to
\(\psi(-t)=-\psi(t)\).  Combining this with \textup{(ii)} gives
\cref{eq:zonal-sign-family}.  Continuity makes the
\(\{-1,1\}\)-valued function \(\varepsilon\) constant on the connected
interval \((0,1)\).
\end{proof}

\begin{corollary}[Hardy orientation selects the positive reciprocal field]
\label{cor:zonal-hardy-orientation}
Assume the equivalent conditions in \cref{thm:zonal-rigidity} and suppose
that \(T_\psi\) is antipodally odd.  Set, on each oriented plane,
\[
  F_P(\phi):=-T_\psi(U_\phi)^\top V_\phi,
  \qquad
  b_P(e^{2i\phi})
  :=
  \frac{1+iF_P(\phi)}{1-iF_P(\phi)}.
\]
If \(b_P\) is the boundary function of a member of \(H^\infty(\D)\) for
Haar-almost every \(P\), then
\[
  T_\psi(u)=\frac{a}{a^\top u}-u
  \quad\text{almost everywhere}.
\]
The converse also holds.
\end{corollary}

\begin{proof}
Use the basis in the proof of \cref{thm:zonal-rigidity} and put
\(z=e^{2i\theta}\).  Antipodal oddness makes \(F_P\) \(\pi\)-periodic, so
\(b_P\) is a well-defined function of \(z\).  The sign representation
\cref{eq:zonal-sign-family} gives
\[
  b_P(z)=
  \begin{cases}
    z,&\varepsilon(r|\cos\theta|)=1,\\
    z^{-1},&\varepsilon(r|\cos\theta|)=-1.
  \end{cases}
\]
If the first set has positive circle measure and \(B\in H^\infty\) has this
boundary function, Privalov uniqueness applied to \(B(z)-z\) gives
\(B(z)=z\); the second set must then be null.  If the first set is null,
\(zB(z)\) has boundary value one almost everywhere, which would force
\(zB(z)=1\), an impossibility at zero.  Thus for every Hardy-good radius
\(r\), \(\varepsilon(r|\cos\theta|)=1\) almost everywhere.  The pushforward
of uniform \(\theta\) under \(r|\cos\theta|\) has a strictly positive
density on \((0,r)\).  Taking good radii increasing to one proves
\(\varepsilon=1\) almost everywhere.  The converse has \(b_P(z)=z\), up to
a rotation of coordinates.  We use the standard boundary uniqueness
theorem for Hardy functions \cite{duren-1970}.
\end{proof}

\begin{remark}
The almost-every-plane hypothesis in \cref{thm:zonal-rigidity} cannot be
replaced by a positive-measure family of planes.  For example, the odd
function
\[
  \psi(t)=\frac{2t}{2t^2-1}
\]
does not have reciprocal magnitude.  Put
\(b=(1-r^2)/r^2\), \(\varphi=2\theta\), and
\[
  \gamma=\arctan\sqrt{\frac{1-b}{1+b}}.
\]
For every \(r\geq1/\sqrt2\), direct trigonometry gives
\[
  Z_r(\theta)
  =-\frac{\sin\varphi}{\cos\varphi-b}
  =-\frac12\tan\!\left(\theta-\gamma-\frac\pi2\right)
   -\frac12\tan\!\left(\theta+\gamma-\frac\pi2\right).
\]
Hence \(Z_r\) is standard Cauchy by \cref{cor:weighted-tangent} and symmetry
throughout the interval \([1/\sqrt2,1]\).  For \(r<1/\sqrt2\) it is bounded and cannot
be Cauchy.  Thus a positive range of projection radii is insufficient; the
small radii present under the \(p\geq3\) Haar-plane law anchor the Abel
inversion.  In dimension \(p=2\), where only \(r=1\) is seen, this example
also shows why the stated higher-dimensional rigidity fails.
\end{remark}

\subsection{The simplex is the maximal universal real-weight class}

\begin{theorem}[Sharp maximality of positive weights]
\label{thm:simplex-maximality}
Fix \(p\geq2\) and \(c=(c_1,\ldots,c_k)\in\R^k\) with
\(\sum_{j=1}^k c_j=1\).  The following are equivalent.
\begin{enumerate}[label=\textup{(\roman*)}]
\item \(c_j\geq0\) for every \(j\).
\item For every collection of nonzero vectors
  \(a_1,\ldots,a_k\in\R^p\), the field
  \[
    g_{a,c}(u):=\sum_{j=1}^k c_j\frac{a_j}{a_j^\top u}
  \]
  satisfies
  \begin{equation}
    \frac{1}{\|g_{a,c}(U)\|^2}
    \sim\Beta\!\left(\frac12,\frac{p-1}{2}\right).
    \label{eq:signed-universal-incidence}
  \end{equation}
\item For every \(\alpha_1,\ldots,\alpha_k\in\R\) and
  \(\Phi\sim\Unif(\R/(\pi\mathbb Z))\),
  \begin{equation}
    \sum_{j=1}^k c_j\tan(\Phi-\alpha_j)
    \sim\Cauchy(0,1).
    \label{eq:signed-universal-tangent}
  \end{equation}
\end{enumerate}
Thus, after the Euler normalization \(\sum_jc_j=1\), the probability simplex
is exactly the maximal real coefficient class for which either universality
holds over every arrangement.
\end{theorem}

\begin{proof}
When \(k=1\), the normalization forces \(c_1=1\), so all three assertions
are immediate.  Hence suppose \(k\geq2\).
If \textup{(i)} holds, \textup{(ii)} is
\cref{thm:universal-beta}, and \textup{(iii)} is
\cref{eq:weighted-tangent-preview}.

Suppose that \textup{(ii)} holds but \(c_r<0\) for some \(r\).
Choose the repeated two-direction arrangement
\[
  a_r=e_1,
  \qquad
  a_j=e_2\quad(j\neq r),
\]
where \(e_1,e_2\) are the first two coordinate vectors in \(\R^p\).
Away from its null set of poles,
\[
  g_{a,c}(U)
  =
  \frac{c_r}{U_1}e_1
  +
  \frac{1-c_r}{U_2}e_2.
\]
Cauchy--Schwarz gives
\begin{align*}
  \|g_{a,c}(U)\|^2
  &=
  \frac{c_r^2}{U_1^2}
  +
  \frac{(1-c_r)^2}{U_2^2}\\
  &\geq
  \frac{\{|c_r|+|1-c_r|\}^2}{U_1^2+U_2^2}
  \geq
  \{|c_r|+|1-c_r|\}^2.
\end{align*}
The final inequality uses \(U_1^2+U_2^2\leq1\).
Because \(c_r<0\),
\[
  |c_r|+|1-c_r|=1-2c_r>1.
\]
Hence the variable in \cref{eq:signed-universal-incidence} is bounded above
by \((1-2c_r)^{-2}<1\).  This contradicts the strictly positive density of
the target Beta law throughout \((0,1)\).  Therefore \textup{(ii)} implies
\textup{(i)}.

Finally, suppose that \textup{(iii)} holds and that at least one coefficient
is negative.  Put
\[
  B:=-\sum_{j:c_j<0}c_j>0,
  \qquad
  A:=\sum_{j:c_j>0}c_j=1+B.
\]
Choose
\[
  \alpha_j=0\quad\text{when }c_j\geq0,
  \qquad
  \alpha_j=\frac{\pi}{2}\quad\text{when }c_j<0.
\]
The choice for a zero coefficient is immaterial.
With \(X=\tan\Phi\), the left side of
\cref{eq:signed-universal-tangent} becomes
\[
  AX+\frac{B}{X}.
\]
For \(X>0\) it is at least \(2\sqrt{AB}\), while for \(X<0\) it is at most
\(-2\sqrt{AB}\).  Its distribution therefore assigns no mass to
\((-2\sqrt{AB},2\sqrt{AB})\), which is impossible for a standard Cauchy
law.  Thus \textup{(iii)} also implies \textup{(i)}.
\end{proof}

\begin{remark}
Without normalizing the coefficients, let \(b_j\in\R\) and
\(d=\sum_jb_j\neq0\).  The corresponding normalized incidence variable is
\[
  \frac{d^2}{
    \left\|\sum_jb_j a_j/(a_j^\top U)\right\|^2}.
\]
It has the universal Beta law for every arrangement if and only if
\(b_j/d\geq0\) for every \(j\).  Thus all nonzero coefficients must have the
same sign as their total; normalization reduces this unnormalized same-sign
class to the simplex in \cref{thm:simplex-maximality}.
\end{remark}

\section{Probability-measure arrangements}
\label{sec:measure-arrangements}

The finite theorem concerns atomic probability measures on the space of
directions.  This section removes atomicity.  Empirical approximation extends
the Beta law to arbitrary Borel probability measures provided the reciprocal
field is absolutely integrable.  The same hypothesis also recovers the
plane-conditional Cauchy mechanism.  For signed measures, the conditional law
determines an exact phase-cancellation deficit; positivity follows whenever
the phase views norm total variation.  This condition is strictly weaker than
injectivity, and a cross-sign product-dimension bound is one checkable
sufficient criterion.

\begin{theorem}[Borel probability-measure incidence law]
\label{thm:measure-beta}
Let \(p\geq2\), let \(\mu\) be a Borel probability measure on
\(\sphere^{p-1}\), and suppose that
\begin{equation}
  I_\mu(U)
  :=
  \int_{\sphere^{p-1}}
  \frac{1}{|a^\top U|}
  \,\mu(\dd a)
  <\infty
  \qquad\text{almost surely}
  \label{eq:measure-reciprocal-integrability}
\end{equation}
for \(U\sim\Unif(\sphere^{p-1})\).  At every \(u\) satisfying the
integrability condition in \cref{eq:measure-reciprocal-integrability}, define
the absolutely convergent Bochner integral
\begin{equation}
  g_\mu(u)
  :=
  \int_{\sphere^{p-1}}
  \frac{a}{a^\top u}
  \,\mu(\dd a),
  \qquad
  N_\mu(u):=\frac{g_\mu(u)}{\|g_\mu(u)\|},
  \qquad
  H_\mu(u):=\frac{1}{\|g_\mu(u)\|^2}.
  \label{eq:measure-reciprocal-field}
\end{equation}
Then \(u^\top g_\mu(u)=1\) wherever the field is defined, and
\begin{equation}
  H_\mu(U)
  =
  \{U^\top N_\mu(U)\}^2
  \sim
  \Beta\!\left(\frac12,\frac{p-1}{2}\right).
  \label{eq:measure-beta}
\end{equation}
\end{theorem}

\begin{proof}
Let \(\sigma\) denote Haar probability measure on
\(\sphere^{p-1}\), and let \(A_1,A_2,\ldots\) be independent with common law
\(\mu\) on an auxiliary probability space.  Extend the Borel kernel
\[
  r(a,u):=\frac{a}{a^\top u}
\]
by setting \(r(a,u)=0\) when \(a^\top u=0\).  If \(u\) satisfies
\cref{eq:measure-reciprocal-integrability}, then
\(\mu\{a:a^\top u=0\}=0\), and
\[
  \E\|r(A_1,u)\|=I_\mu(u)<\infty,
  \qquad
  \E r(A_1,u)=g_\mu(u).
\]
Define the set of full Haar measure
\[
  D_\mu:=\{u:I_\mu(u)<\infty\}.
\]
For every \(u\in D_\mu\), the corresponding pole set has
\(\mu\)-measure zero.  Hence
\begin{equation}
  u^\top g_\mu(u)
  =
  \int_{\{a:a^\top u\neq0\}}
  \frac{u^\top a}{a^\top u}\,\mu(\dd a)
  =1.
  \label{eq:measure-euler-pairing}
\end{equation}
The set \(D_\mu\) is Borel: the extended nonnegative kernel
\((a,u)\mapsto |a^\top u|^{-1}\), taking value \(+\infty\) when
\(a^\top u=0\), is jointly Borel, and its parameter integral is Borel.  The
same observation component by component shows that
\(g_\mu\) is Borel on \(D_\mu\).  Extend it by zero on \(D_\mu^c\) for the
measurability argument below.
The vector-valued strong law of large numbers therefore gives
\begin{equation}
  \frac1n\sum_{j=1}^n r(A_j,u)
  \longrightarrow g_\mu(u)
  \qquad\text{almost surely}
  \label{eq:empirical-field-slln}
\end{equation}
for every such fixed \(u\).

The set on which the convergence in
\cref{eq:empirical-field-slln} holds is jointly measurable in the sampled
sequence and in \(u\).  Indeed, coordinatewise it is a countable
intersection over \(m\geq1\), followed by a countable union over \(N\), of
the intersections over \(n\geq N\) of the Borel inequalities
\[
  \left|
    \frac1n\sum_{j=1}^n r_\ell(A_j,u)-g_{\mu,\ell}(u)
  \right|<\frac1m .
\]
The hypothesis and Fubini's theorem consequently
show that, for \(\mu^{\otimes\mathbb N}\)-almost every realization
\((a_j)_{j\geq1}\), the convergence holds for
\(\sigma\)-almost every \(u\).  Fix one such deterministic realization.  The
countable union
\[
  \bigcup_{j\geq1}\{u:a_j^\top u=0\}
\]
is also \(\sigma\)-null.  Hence, for \(\sigma\)-almost every \(u\),
\begin{equation}
  g_n(u)
  :=
  \frac1n\sum_{j=1}^n\frac{a_j}{a_j^\top u}
  \longrightarrow g_\mu(u).
  \label{eq:empirical-field-convergence}
\end{equation}

For every \(n\), \cref{thm:universal-beta}, applied to the finite
arrangement \(a_1,\ldots,a_n\) with equal weights, gives
\begin{equation}
  H_n(U):=\frac{1}{\|g_n(U)\|^2}
  \sim\Beta\!\left(\frac12,\frac{p-1}{2}\right).
  \label{eq:empirical-beta}
\end{equation}
Moreover, \cref{eq:measure-euler-pairing} gives
\(u^\top g_\mu(u)=1\), so \(g_\mu(u)\neq0\), and
\[
  H_n(U)\longrightarrow H_\mu(U)
  \qquad\text{almost surely}.
\]
Taking expectations of bounded continuous functions and using
\cref{eq:empirical-beta} proves \cref{eq:measure-beta}.  Finally,
\(U^\top N_\mu(U)=1/\|g_\mu(U)\|\), which gives the incidence form.
\end{proof}

The condition in \cref{eq:measure-reciprocal-integrability} is scale-free in
the directions \(a\).  Thus an equivalent statement may be formulated for
a probability measure on real projective space by choosing arbitrary unit
representatives.

\subsection{Absolute integrability forces singular phase projections}

The empirical proof gives more than a closure theorem: its integrability
hypothesis forces the one-dimensional phase measures seen from almost every
two-plane to be singular.  The common-phase argument then extends through a
Herglotz integral.

\begin{proposition}[Conditional Cauchy law for measure arrangements]
\label{prop:measure-conditional-cauchy}
Under the assumptions of \cref{thm:measure-beta}, let \((U,V)\) be a Haar
oriented orthonormal two-frame and let \(P=\operatorname{span}\{U,V\}\) with
its induced orientation.  For Haar-almost every oriented plane \(P\), the
angular pushforward modulo \(\pi\) of \(\mu\) under orthogonal projection
onto \(P\) is a singular probability measure, and
\begin{equation}
  \Law\{g_\mu(U)^\top V\mid P\}
  =\Cauchy(0,1).
  \label{eq:measure-conditional-cauchy}
\end{equation}
Consequently, \cref{thm:measure-beta} also has a direct conditional
common-phase proof.
\end{proposition}

\begin{proof}
Fix an oriented orthonormal basis \((e_1,e_2)\) of \(P\) and write
\begin{equation}
  U_\phi=e_1\cos\phi+e_2\sin\phi,
  \qquad
  V_\phi=-e_1\sin\phi+e_2\cos\phi.
  \label{eq:measure-frame-phase}
\end{equation}
For a fixed \(a\in\sphere^{p-1}\), the event that \(a\perp P\) has
Grassmannian probability zero.  Fubini's theorem therefore gives
\begin{equation}
  \mu(P^\perp\cap\sphere^{p-1})=0
  \quad\text{for almost every }P.
  \label{eq:measure-no-orthogonal-mass}
\end{equation}
For such a plane, write
\begin{equation}
  \operatorname{proj}_P(a)
  =r_P(a)
   \{e_1\cos\widetilde\alpha_P(a)
      +e_2\sin\widetilde\alpha_P(a)\},
  \qquad 0<r_P(a)\leq1,
  \label{eq:measure-phase-map}
\end{equation}
where \(\widetilde\alpha_P(a)\) is taken modulo \(2\pi\).  Let
\(\alpha_P(a)\) be its class modulo \(\pi\), define \(\alpha_P\)
arbitrarily on the \(\mu\)-null set
\(P^\perp\cap\sphere^{p-1}\), and let
  \begin{equation}
  \nu_P=(\alpha_P)_\#\mu,
  \qquad
  \eta_P(B)
  :=
  \int
  \frac{\boldsymbol 1\{\alpha_P(a)\in B\}}{r_P(a)}
  \,\mu(\dd a).
  \label{eq:weighted-phase-measure}
  \end{equation}
Thus \(\nu_P\) is a Borel probability measure.  The weighted phase measure
\(\eta_P\) is the exact object governing absolute reciprocal
integrability.

The Haar disintegration in \cref{lem:uniform-frame-phase} and the hypothesis
\cref{eq:measure-reciprocal-integrability} imply that, for almost every
\(P\), \(I_\mu(U_\phi)<\infty\) for almost every \(\phi\).  For those
\((P,\phi)\),
\begin{equation}
  I_\mu(U_\phi)
  =
  \int
  \frac{1}{r_P(a)|\cos\{\phi-\alpha_P(a)\}|}
  \,\mu(\dd a)
  =
  J_{\eta_P}(\phi).
  \label{eq:phase-secants-controlled}
\end{equation}
The identity in \cref{eq:phase-secants-controlled} is purely algebraic and
holds whenever the two extended integrals are defined; the hypothesis of
\cref{thm:measure-beta} is used here only to obtain almost-everywhere
finiteness.
Because \(J_{\eta_P}\) is finite at almost every phase, \(\eta_P\) is
finite and \cref{lem:secant-singularity} makes it singular.  The measures
\(\eta_P\) and \(\nu_P\) have the same null sets: for every Borel \(B\),
the integrand defining \(\eta_P(B)\) is finite and strictly positive on
\(\alpha_P^{-1}(B)\), so either measure vanishes on \(B\) exactly when the
other does.  Hence \(\nu_P\) is singular as well.  Moreover,
\(J_{\nu_P}\leq J_{\eta_P}<\infty\) almost everywhere, so
\cref{thm:planar-cauchy} applies to the probability measure \(\nu_P\).

Absolute convergence and
\cref{eq:measure-frame-phase,eq:measure-phase-map} give
\begin{align}
  g_\mu(U_\phi)^\top V_\phi
  &=
  \int\frac{a^\top V_\phi}{a^\top U_\phi}\,\mu(\dd a)
  \notag\\
  &=-\int\tan(\phi-\alpha)\,\nu_P(\dd\alpha).
  \label{eq:measure-tangent-boundary}
\end{align}
The right side without its leading minus sign is standard Cauchy by
\cref{thm:planar-cauchy}; symmetry proves
\cref{eq:measure-conditional-cauchy}.

Finally, \(T_\mu(U):=g_\mu(U)-U\) is tangent and
\(T_\mu(U)^\top V=g_\mu(U)^\top V\).  Applying
\cref{thm:tangent-projection-equivalence} gives a second proof of
\cref{thm:measure-beta}.
\end{proof}

\subsection{A local-to-global signed-measure converse}

For a finite real signed measure \(\eta\) on \(\mathbb{RP}^{p-1}\), the
projective phase pushforward is defined for Haar-almost every oriented
two-plane by
\begin{equation}
  \nu_P:=(\alpha_P)_\#\eta.
  \label{eq:signed-phase-pushforward}
\end{equation}
Indeed, for every fixed projective direction, the planes onto which it
projects to zero form a Haar-null set, so Fubini's theorem applies to
\(|\eta|\).  Define the essential phase norm and its cancellation deficit by
\begin{equation}
  \|\eta\|_{\mathrm{ph}}
  :=\operatorname*{ess\,sup}_{P\in\operatorname{Gr}^+(2,p)}
       \|\nu_P\|_{\mathrm{TV}},
  \qquad
  \Delta_{\mathrm{ph}}(\eta)
  :=\|\eta\|_{\mathrm{TV}}-\|\eta\|_{\mathrm{ph}}\geq0.
  \label{eq:phase-norm-deficit}
\end{equation}
The inequality is contraction of total variation under pushforward.  We
use any positive orthonormal basis of \(P\); changing it only rotates the
phase circle and leaves the norm unchanged.  We
call \(\eta\) \emph{phase norming} when
\(\|\eta\|_{\mathrm{ph}}=\|\eta\|_{\mathrm{TV}}\).  If
\(\eta=\eta^+-\eta^-\) is the Jordan decomposition, a fixed phase view
preserves total variation exactly when
\((\alpha_P)_\#\eta^+\) and \((\alpha_P)_\#\eta^-\) are mutually singular.
Thus phase norming asks only for arbitrarily small essential phase
cancellation.  It neither requires an injective phase map nor forbids
collisions between points carrying the same sign.

\begin{remark}[A single view can hide signs]
Let \(p\geq3\), \(P_0=\operatorname{span}(e_1,e_2)\),
\(a_\pm=(e_1\pm e_3)/\sqrt2\), and, for \(c>0\),
\[
  \eta=(1+c)\delta_{[a_+]}-c\delta_{[a_-]}.
\]
This is a mass-one signed measure with total variation \(1+2c\), but the two
directions have the same \(P_0\)-phase, so \(\nu_{P_0}=\delta_0\).  That view
is positive and its tangent boundary law is standard Cauchy.  For
Haar-almost every plane, however, the two phases are distinct and the
pushforward retains total variation \(1+2c\).  Thus one favorable view can be
misleading; the almost-every-plane law and the essential phase norm capture
different information from any selected projection.
\end{remark}

\begin{theorem}[Phase-norming local-to-global rigidity]
\label{thm:thin-support-rigidity}
Let \(p\geq2\), let \(\eta\) be a finite real signed measure of mass one on
\(\mathbb{RP}^{p-1}\), and suppose that
\begin{equation}
  I_{|\eta|}(U)
  :=\int_{\mathbb{RP}^{p-1}}
      \frac{1}{|a^\top U|}\,|\eta|(\dd[a])
  <\infty
  \qquad\text{almost surely},
  \label{eq:signed-reciprocal-integrability}
\end{equation}
where either unit representative of \([a]\) may be used.  Define
\begin{equation}
  g_\eta(u)
  :=\int\frac{a}{a^\top u}\,\eta(\dd[a]),
  \qquad
  T_\eta(u):=g_\eta(u)-u.
  \label{eq:signed-projective-field}
\end{equation}
If
\begin{equation}
  \Law\{T_\eta(U)^\top V\mid P\}=\Cauchy(0,1)
  \quad\text{for Haar-almost every oriented plane }P,
  \label{eq:signed-conditional-cauchy}
\end{equation}
then \(\nu_P\) is a positive probability measure for Haar-almost every
\(P\).  Equivalently, if \(\eta=\eta^+-\eta^-\), then
\begin{equation}
  (\alpha_P)_\#\eta^-
  \leq
  (\alpha_P)_\#\eta^+
  \quad\text{for Haar-almost every }P.
  \label{eq:phase-jordan-domination}
\end{equation}
Moreover,
\begin{equation}
  \|\eta\|_{\mathrm{ph}}=1,
  \qquad
  \Delta_{\mathrm{ph}}(\eta)
  =2\eta^-(\mathbb{RP}^{p-1}).
  \label{eq:phase-deficit-negative-mass}
\end{equation}
Consequently, within the phase-norming class, the following are equivalent:
\begin{enumerate}[label=\textup{(\roman*)}]
\item \(\eta\) is a Borel probability measure;
\item the conditional Cauchy law \cref{eq:signed-conditional-cauchy} holds.
\end{enumerate}
For \(p=2\), every signed measure is phase norming, so no additional phase
hypothesis is needed.
\end{theorem}

\begin{proof}
The integrand in \cref{eq:signed-projective-field} does not depend on the
choice of representative, and \cref{eq:signed-reciprocal-integrability}
makes the field absolutely defined almost surely.  Haar-frame
disintegration shows that, for almost every oriented plane \(P\),
\(I_{|\eta|}(U_\phi)<\infty\) for almost every phase \(\phi\).  For such a
plane, let
\[
  \lambda_P:=(\alpha_P)_\#|\eta|,
  \qquad
  \xi_P(B):=\int
    \frac{\boldsymbol 1\{\alpha_P(a)\in B\}}{r_P(a)}
    \,|\eta|(\dd[a]),
\]
using the projected radius in \cref{eq:measure-phase-map}.  Exactly as in
\cref{eq:phase-secants-controlled},
\[
  J_{\xi_P}(\phi)=I_{|\eta|}(U_\phi)<\infty
  \quad\text{for almost every }\phi.
\]
Because the secant kernel is at least one, finiteness at any such phase
implies that \(\xi_P\) is finite; it is then singular by
\cref{lem:secant-singularity}.  Since
\(\lambda_P\leq\xi_P\) and \(|\nu_P|\leq\lambda_P\), the signed phase
measure \(\nu_P\) is singular as well, and
\(J_{|\nu_P|}\leq J_{\xi_P}<\infty\) almost everywhere.

Absolute integration gives
\begin{equation}
  -T_\eta(U_\phi)^\top V_\phi
  =\int_{\mathbb T_\pi}
      \tan(\phi-\alpha)\,\nu_P(\dd\alpha).
  \label{eq:signed-phase-tangent}
\end{equation}
Indeed, the absolute value of the original ratio is bounded by
\(|a^\top U_\phi|^{-1}\).  Thus
\cref{lem:signed-ordinary-boundary} identifies
\cref{eq:signed-phase-tangent} with the Herglotz boundary value.  Under
\cref{eq:signed-conditional-cauchy}, Cauchy symmetry and
\cref{thm:signed-clark-rigidity} show that \(\nu_P\) is positive for almost
every \(P\).  Its mass is one, so \(\|\nu_P\|_{\mathrm{TV}}=1\), which
proves \(\|\eta\|_{\mathrm{ph}}=1\).  Since a mass-one signed measure
satisfies
\[
  \|\eta\|_{\mathrm{TV}}
  =1+2\eta^-(\mathbb{RP}^{p-1}),
\]
\cref{eq:phase-deficit-negative-mass} follows.

If \(\eta\) is phase norming, the conditional law therefore gives
\(\|\eta\|_{\mathrm{TV}}=1=\eta(\mathbb{RP}^{p-1})\), and the Jordan
decomposition forces \(\eta^-\) to vanish.  Conversely, if \(\eta\) is a
probability measure, it is automatically phase norming.  After choosing a
measurable spherical lift, \cref{prop:measure-conditional-cauchy} proves the
conditional law.

When \(p=2\), the only two-plane is the ambient plane and its projective
phase map is a projective isomorphism.  Its pushforward preserves total
variation for every signed measure.
\end{proof}

The phase-norming condition is strictly weaker than injectivity of a phase
view.  The next result gives a geometric sufficient criterion involving only
opposite signs.

\begin{lemma}[Generic cross-sign separation below the collision threshold]
\label{lem:generic-cross-sign-separation}
Let \(p\geq3\), and let \(K_+,K_-\subset\mathbb{RP}^{p-1}\) be disjoint
nonempty compact sets satisfying
\begin{equation}
  \dim_{\mathrm H}(K_+\times K_-)<1.
  \label{eq:cross-product-dimension-threshold}
\end{equation}
For Haar-almost every oriented two-plane \(P\), no point of
\(K_+\cup K_-\) projects to zero and
\begin{equation}
  \alpha_P(K_+)\cap\alpha_P(K_-)=\varnothing.
  \label{eq:cross-sign-phase-separation}
\end{equation}
\end{lemma}

\begin{proof}
Let \(G=\operatorname{Gr}^+(2,p)\), of dimension \(n=2(p-2)\).  Since
both factors are nonempty, their embeddings into the product show that
\(\dim_{\mathrm H}K_+<1\) and \(\dim_{\mathrm H}K_-<1\).  For either
carrier, consider the incidence pairs \(([a],P)\) with \(P\subset a^\perp\).
The fiber over \([a]\) is \(\operatorname{Gr}^+(2,p-1)\), of dimension
\(2(p-3)\).  Local smooth trivializations therefore bound the dimension of
the corresponding exceptional plane set by
\[
  \dim_{\mathrm H}K_\pm+2(p-3)<2(p-2)=n.
\]
Thus almost every plane has nonzero projection for every point of both
carriers.

The compact sets are disjoint, so their projective distance is positive.
For a fixed pair \(([a],[b])\in K_+\times K_-\), equality of projected
phases is equivalent to the existence of a projective line
\([x]\subset\operatorname{span}(a,b)\) for which \(P\subset x^\perp\).
The corresponding flag incidence
\[
  \{([a],[b],[x],P):[x]\subset\operatorname{span}(a,b),
       \ P\subset x^\perp\}
\]
is locally a smooth bundle over the separated pair space with fiber
dimension
\(1+2(p-3)=n-1\).  Restricting the base to \(K_+\times K_-\) therefore
gives Hausdorff dimension at most
\[
  \dim_{\mathrm H}(K_+\times K_-)+n-1<n.
\]
Its projection to \(G\) is Haar-null, proving
\cref{eq:cross-sign-phase-separation}.
\end{proof}

\begin{corollary}[Cross-sign carrier criterion]
\label{cor:cross-sign-carrier-rigidity}
Let \(\eta\) satisfy the mass-one and reciprocal-integrability assumptions
of \cref{thm:thin-support-rigidity}, and write
\(\eta=\eta^+-\eta^-\).  Suppose the two Jordan parts are carried by
disjoint \(\sigma\)-compact sets \(E_+,E_-\subset\mathbb{RP}^{p-1}\) such that
\begin{equation}
  \dim_{\mathrm H}(E_+\times E_-)<1.
  \label{eq:cross-sign-carrier-threshold}
\end{equation}
Then \(\eta\) is phase norming, and hence the conditional Cauchy law holds
if and only if \(\eta\) is a probability measure.

It is sufficient for \cref{eq:cross-sign-carrier-threshold} that either
\[
  \dim_{\mathrm H}E_+ + \dim_{\mathrm P}E_-<1
  \quad\text{or}\quad
  \dim_{\mathrm P}E_+ + \dim_{\mathrm H}E_-<1.
\]
In particular, the conclusion recovers the former full-support criterion
\(\dim_{\mathrm H}(K\times K)<1\), and hence the sufficient condition
\(\dim_{\mathrm P}K<1/2\), when \(|\eta|\) is carried by one compact set
\(K\).
\end{corollary}

\begin{proof}
The assertion is immediate if one Jordan part vanishes.  If \(p=2\), phase
norming is already part of \cref{thm:thin-support-rigidity}, so suppose
\(p\geq3\).  Borel regularity supplies such disjoint \(\sigma\)-compact
carriers inside the Hahn sets, in particular when both parts are carried by
one compact \(K\).  Exhaust
each \(\sigma\)-compact carrier by an increasing sequence of compact sets.  Apply
\cref{lem:generic-cross-sign-separation} to every pair in the two
exhaustions and discard the resulting countable union of exceptional plane
sets.  For every remaining plane, the two phase pushforwards are carried by
disjoint Borel sets, each a countable union of compact images.  They are
therefore mutually singular, so phase pushforward preserves the total
variation of \(\eta\).  This proves phase norming and allows
\cref{thm:thin-support-rigidity} to be applied.

The two displayed sufficient conditions follow from Tricot's product
inequality
\[
  \dim_{\mathrm H}(E\times F)
  \leq\dim_{\mathrm H}E+\dim_{\mathrm P}F
\]
and its version with the factors reversed
\cite{tricot-1982,falconer-2014}.
\end{proof}

The cross-sign condition can be genuinely weaker than the former
full-support condition.  For example, take disjoint self-similar Cantor
carriers of dimensions \(s,t\in(0,1)\) with
\(s+t<1\leq2\max\{s,t\}\).  Their cross product is subcritical, while the
Cartesian square of their union is not.

Even without phase norming, the conditional law already excludes atomic
negative mass.

\begin{corollary}[No negative atoms]
\label{cor:no-negative-atoms}
Under \cref{eq:signed-reciprocal-integrability,eq:signed-conditional-cauchy},
the negative Jordan part \(\eta^-\) is atomless.  Consequently, the full
positivity conclusion holds whenever \(\eta^-\) is purely atomic.
\end{corollary}

\begin{proof}
If \(\eta^-\) had an atom of mass \(c>0\) at \([a_0]\), then
\(\eta\{[a_0]\}=-c\).  For every fixed \([b]\neq[a_0]\), the projected
phases of \([a_0]\) and \([b]\) collide only for a Haar-null set of
planes.  Fubini's theorem applied to
\(|\eta|\) off \([a_0]\) shows that, for almost every plane, the phase
fiber through \(\alpha_P(a_0)\) contains no other \(|\eta|\)-mass.  Thus
\(\nu_P\) has an atom of mass \(-c\), contradicting the almost-every-plane
positivity furnished by \cref{thm:thin-support-rigidity}.
\end{proof}

\begin{remark}
The exact unresolved issue is now diffuse phase cancellation, not geometric
injectivity.  Under the conditional Cauchy law,
\cref{eq:phase-jordan-domination,eq:phase-deficit-negative-mass} say that
the negative pushforward is dominated in almost every phase view and that
the phase deficit records
twice the hidden negative mass.  It remains open whether the reciprocal
integrability and compatibility across all planes force this deficit to
vanish.  The cross-sign criterion proves that they do whenever opposite
Jordan carriers fall below the collision threshold, even if either carrier
has extensive same-sign collisions.  Hausdorff dimension alone does not
control Cartesian products; see Besicovitch--Moran
\cite{besicovitch-moran-1945}.
\end{remark}

\section{Integrability boundaries for infinite arrangements}
\label{sec:countable-boundary}

Absolute reciprocal integrability is the precise gate in
\cref{thm:measure-beta}: without it, the vector field need not exist as an
ordinary integral.  This section analyzes that gate rather than introducing a
different phenomenon.  It first gives an exact pointwise criterion in terms of
mass accumulating near orthogonal hyperplanes, then translates the criterion
into sharp weight, clustering, and support conditions.

\subsection{An exact orthogonality-profile criterion}

For \(u\in\sphere^{p-1}\) and \(0<t\leq1\), define the mass of the
orthogonal belt of width \(t\) by
\begin{equation}
  b_\mu(u,t)
  :=
  \mu\{a:|a^\top u|\leq t\}.
  \label{eq:orthogonality-profile}
\end{equation}

\begin{theorem}[Wiener--Dini incidence criterion]
\label{thm:wiener-incidence}
For every finite Borel measure \(\mu\) on \(\sphere^{p-1}\), write
\(M_\mu=\mu(\sphere^{p-1})\).  Then, for every
\(u\in\sphere^{p-1}\), with extended values allowed,
\begin{equation}
  I_\mu(u)
  =
  M_\mu+\int_0^1\frac{b_\mu(u,t)}{t^2}\,\dd t.
  \label{eq:wiener-profile-integral}
\end{equation}
If
\begin{equation}
  W_\mu(u)
  :=
  \sum_{n=0}^\infty
  2^n b_\mu(u,2^{-n}),
  \label{eq:wiener-profile-series}
\end{equation}
then
\begin{equation}
  \frac{W_\mu(u)+M_\mu}{2}
  \leq I_\mu(u)\leq W_\mu(u)+M_\mu.
  \label{eq:wiener-profile-comparison}
\end{equation}
In particular, \cref{eq:measure-reciprocal-integrability} holds if and only
if the series in \cref{eq:wiener-profile-series} is finite for
Haar-almost every \(u\).
\end{theorem}

\begin{proof}
For \(0\leq x\leq1\),
\[
  \frac1x
  =
  1+\int_0^1
  \boldsymbol 1\{x\leq t\}\frac{\dd t}{t^2},
\]
with both sides interpreted as infinity when \(x=0\).  Apply this identity
to \(x=|a^\top u|\), integrate in \(a\), and use Tonelli's theorem to
obtain \cref{eq:wiener-profile-integral}.  Monotonicity of
\(t\mapsto b_\mu(u,t)\) gives, on
\((2^{-(n+1)},2^{-n}]\),
\[
  2^n b_\mu(u,2^{-(n+1)})
  \leq
  \int_{2^{-(n+1)}}^{2^{-n}}
  \frac{b_\mu(u,t)}{t^2}\,\dd t
  \leq
  2^n b_\mu(u,2^{-n}).
\]
Sum over \(n\geq0\) and use \(b_\mu(u,1)=M_\mu\).
\end{proof}

\begin{corollary}[Pointwise projection tests]
\label{cor:projection-tests}
At a fixed \(u\), finiteness of \(I_\mu(u)\) forces
\[
  b_\mu(u,t)=o(t)\qquad(t\downarrow0).
\]
It is sufficient that, for some \(C(u)<\infty\) and \(\epsilon(u)>0\),
\[
  b_\mu(u,t)\leq C(u)t^{1+\epsilon(u)}
\]
for all sufficiently small \(t\).  Equivalently,
\cref{eq:wiener-profile-integral} is the Dini condition at zero for the
one-dimensional projection
\((a\mapsto a^\top u)_\#\mu\).
\end{corollary}

\begin{proof}
The summands in \cref{eq:wiener-profile-series} must tend to zero.  Dyadic
monotonicity extends this to \(b_\mu(u,t)/t\to0\).  The power bound makes
the series geometrically summable.
\end{proof}

\begin{corollary}[Exact plane-wise criterion]
\label{cor:plane-wise-wiener}
For the weighted phase measure \(\eta_P\) in the proof of
\cref{prop:measure-conditional-cauchy}, one has, for almost every oriented
plane \(P\),
\[
  I_\mu(U_\phi)=J_{\eta_P}(\phi).
\]
Moreover,
\cref{eq:measure-reciprocal-integrability} holds if and only if
\begin{equation}
  \sum_{n=0}^\infty 2^n
  \eta_P\{\alpha:|\cos(\phi-\alpha)|\leq2^{-n}\}
  <\infty
  \label{eq:plane-wise-wiener}
\end{equation}
for almost every pair \((P,\phi)\).
\end{corollary}

\begin{proof}
For a fixed unit vector \(a\), if
\(r_P(a)=\|\operatorname{proj}_P(a)\|\), then
\[
  \E_P\frac1{r_P(a)}
  =
  \frac{\sqrt\pi\,\Gamma(p/2)}{\Gamma((p-1)/2)}
  <\infty
\]
for \(p>2\), while \(r_P(a)=1\) for \(p=2\).  Tonelli's theorem therefore
makes \(\eta_P\) finite for almost every \(P\).  The first identity is
\cref{eq:phase-secants-controlled}.  Apply the dyadic proof of
\cref{thm:wiener-incidence} to the finite measure \(\eta_P\), and then use
Haar-frame disintegration.
\end{proof}

\begin{remark}[Why a first-moment test cannot classify admissibility]
\label{rem:wiener-fubini-obstruction}
Spherical invariance gives, for every \(\mu\),
\[
  \int b_\mu(u,t)\,\sigma(\dd u)
  =
  \Prb\{|U_1|\leq t\}\sim c_p t
  \qquad(t\downarrow0).
\]
Consequently \(\int W_\mu\,\dd\sigma=\infty\) for every \(\mu\), even for
a point mass, although a point mass has \(I_\mu(U)<\infty\) almost surely.
Thus any intrinsic classification beyond
\cref{thm:wiener-incidence} must retain spatial concentration; Fubini at
the first-moment endpoint necessarily loses it.
\end{remark}

\subsection{The sharp entropy threshold}

For discrete arrangements, universal absolute convergence admits an exact
characterization involving only the weights.  For context on almost-sure
convergence and tail asymptotics of infinite randomly weighted sums with
regularly varying summands, compare \cite{hazra-maulik-2012}.  The point
here is the exact entropy criterion with a universal quantifier over every
deterministic spherical arrangement.

\begin{theorem}[Universal convergence if and only if entropy is finite]
\label{thm:entropy-threshold}
Let \(p\geq2\), let \(w_j>0\) with \(\sum_{j\geq1}w_j=1\), and let
\(U\sim\Unif(\sphere^{p-1})\).  The following are equivalent.
\begin{enumerate}[label=\textup{(\roman*)}]
\item The Shannon entropy of the weights is finite:
  \begin{equation}
    \sum_{j\geq1}w_j\log\frac{e}{w_j}<\infty.
    \label{eq:finite-weight-entropy}
  \end{equation}
\item For every deterministic sequence
  \((a_j)_{j\geq1}\subset\sphere^{p-1}\),
  \begin{equation}
    \sum_{j\geq1}\frac{w_j}{|a_j^\top U|}<\infty
    \qquad\text{almost surely}.
    \label{eq:universal-countable-convergence}
  \end{equation}
\end{enumerate}
If \cref{eq:finite-weight-entropy} fails and the directions \(A_j\) are
independent Haar points on \(\sphere^{p-1}\), then for almost every realized
sequence \((A_j)_{j\geq1}\), the series in
\cref{eq:universal-countable-convergence} diverges for Haar-almost every
\(U\).  Thus the entropy condition is sharp even among generic arrangements.
\end{theorem}

\begin{proof}
Zero weights, if present in an alternative indexing, may simply be
discarded.
Let \(S=|U_1|\).  Its density on \((0,1)\) is
\begin{equation}
  f_p(s)
  =
  \frac{2\Gamma(p/2)}{
    \sqrt\pi\,\Gamma((p-1)/2)}
  (1-s^2)^{(p-3)/2}.
  \label{eq:absolute-coordinate-density}
\end{equation}
For \(0<w\leq1\), put
\[
  m_p(w):=\E\min\left\{\frac{w}{S},1\right\}.
\]
Splitting the integral at \(w\) gives
\begin{equation}
  m_p(w)
  =
  \int_0^w f_p(s)\,\dd s
  +w\int_w^1\frac{f_p(s)}s\,\dd s.
  \label{eq:truncated-reciprocal-moment}
\end{equation}
At \(w=1\), this formula reads \(m_p(1)=1\).  When \(p=2\), the density has
an integrable endpoint singularity at one, so the second integral is still
well defined for every \(0<w<1\).
The density in \cref{eq:absolute-coordinate-density} is bounded above and
away from zero on every sufficiently short interval starting at zero, while
its remaining mass is finite.  Consequently there exist constants
\(0<c_p\leq C_p<\infty\) such that
\begin{equation}
  c_pw\log\frac{e}{w}
  \leq m_p(w)\leq
  C_pw\log\frac{e}{w},
  \qquad 0<w\leq1.
  \label{eq:entropy-comparison}
\end{equation}

Suppose first that \cref{eq:finite-weight-entropy} holds and fix any
deterministic sequence \((a_j)\).  By spherical invariance,
\(|a_j^\top U|\overset{d}{=}S\) for every \(j\).  Tonelli's theorem and
\cref{eq:entropy-comparison} give
\[
  \E\sum_{j\geq1}
  \min\left\{\frac{w_j}{|a_j^\top U|},1\right\}
  =\sum_{j\geq1}m_p(w_j)<\infty.
\]
The sum of the truncated terms is therefore finite almost surely.  Only
finitely many untruncated terms can exceed one, and the countable collection
of poles is Haar-null, proving
\cref{eq:universal-countable-convergence}.

Conversely, suppose \cref{eq:finite-weight-entropy} fails and let
\(A_1,A_2,\ldots\) be independent Haar directions.  Conditional on any
fixed \(u\), the random variables
\[
  Y_j:=\frac{w_j}{|A_j^\top u|}
\]
are independent and satisfy, by \cref{eq:entropy-comparison},
\[
  \sum_{j\geq1}\E\min\{Y_j,1\}=\infty.
\]
For \(y\geq0\),
\begin{equation}
  (1-e^{-1})\min\{y,1\}
  \leq1-e^{-y}\leq\min\{y,1\}.
  \label{eq:laplace-truncation-bound}
\end{equation}
Taking expectations in \cref{eq:laplace-truncation-bound} gives
\[
  1-\E e^{-Y_j}
  \geq(1-e^{-1})\E\min\{Y_j,1\}.
\]
Thus \(\sum_j(1-\E e^{-Y_j})=\infty\).  Independence and
\(1-x\leq e^{-x}\) now give
\[
  \E\exp\left\{-\sum_{j=1}^nY_j\right\}
  =\prod_{j=1}^n\E e^{-Y_j}
  \longrightarrow0.
\]
Dominated convergence then gives
\(\E\exp\{-\sum_{j\geq1}Y_j\}=0\), which forces
\(\sum_jY_j=\infty\) almost surely for each fixed \(u\).  Joint measurability
and Fubini's theorem show that this
divergence holds for Haar-almost every \(u\) for almost every sampled
sequence \((A_j)\).  Fixing one such sequence contradicts \textup{(ii)}, and
also proves the final assertion.
\end{proof}

\begin{corollary}[Countable incidence law]
\label{cor:countable-beta}
Under \cref{eq:finite-weight-entropy}, every deterministic sequence
\((a_j)_{j\geq1}\subset\sphere^{p-1}\) defines, for Haar-almost every \(u\),
the absolutely convergent field
\begin{equation}
  g_{a,w}(u)
  :=\sum_{j\geq1}w_j\frac{a_j}{a_j^\top u}.
  \label{eq:countable-field}
\end{equation}
It satisfies \(u^\top g_{a,w}(u)=1\), and
\begin{equation}
  \frac1{\|g_{a,w}(U)\|^2}
  \sim\Beta\!\left(\frac12,\frac{p-1}{2}\right).
  \label{eq:countable-beta}
\end{equation}
\end{corollary}

\begin{proof}
Apply \cref{thm:measure-beta,thm:entropy-threshold} to the probability
measure \(\mu=\sum_{j\geq1}w_j\delta_{a_j}\).
\end{proof}

\subsection{Entropy--geometry gluing}

\begin{theorem}[Clustered integrability criterion]
\label{thm:clustered-integrability}
Suppose
\[
  \mu=w_0\mu_0+\sum_{j\geq1}w_j\mu_j,
  \qquad
  w_j\geq0,\quad \sum_{j\geq0}w_j=1,
\]
where each \(\mu_j\) is supported on a compact set \(K_j\) whose
orthogonality shadow is Haar-null.  Assume that every \(\mu_j\) is a Borel
probability measure.  For \(j\geq1\), suppose further that
\[
  K_j\subset B(a_j,\rho_j),
  \qquad
  0<\rho_j\leq\frac14,
\]
where the balls use chordal distance, and that
\begin{equation}
  \sum_{j\geq1}\rho_j<\infty,
  \qquad
  \sum_{j\geq1}w_j\log\frac{e}{\rho_j}<\infty.
  \label{eq:clustered-summability}
\end{equation}
Then \(I_\mu(U)<\infty\) almost surely, and consequently \(\mu\) satisfies
the incidence law in \cref{thm:measure-beta}.
\end{theorem}

\begin{proof}
The restriction \(\rho_j\leq1/4\) keeps both the belt estimate and the
truncated logarithmic moment below in their uniform small-radius range.
Let
\[
  E_j:=\{|a_j^\top U|\leq2\rho_j\}.
\]
The spherical coordinate density gives
\(\Prb(E_j)\leq C_p\rho_j\).  The first condition in
\cref{eq:clustered-summability} and the first Borel--Cantelli lemma show
that almost surely only finitely many \(E_j\) occur.  On \(E_j^c\), every
\(a\in K_j\) satisfies
\[
  |a^\top U|
  \geq |a_j^\top U|-\rho_j
  \geq\frac12|a_j^\top U|,
\]
and hence
\[
  I_{\mu_j}(U)\boldsymbol 1_{E_j^c}
  \leq
  \frac{2}{|a_j^\top U|}\boldsymbol 1_{E_j^c}.
\]
If \(S=|U_1|\), then uniformly for \(0<\rho\leq1/4\),
\[
  \E\left\{\frac1S\boldsymbol 1_{\{S>2\rho\}}\right\}
  \leq C_p\log\frac e\rho.
\]
Tonelli's theorem and the second condition in
\cref{eq:clustered-summability} therefore make the sum of all far-cluster
contributions finite almost surely.

Outside the countable union of the null shadows
\(\mathcal O(K_j)\), every individual \(I_{\mu_j}(U)\) is finite.
The \(w_0\mu_0\) term and the finitely many near-cluster terms are therefore
finite as well.  This proves \(I_\mu(U)<\infty\).
\end{proof}

\begin{remark}
The criterion combines two endpoints.  Taking
\(\mu_j=\delta_{a_j}\) and \(\rho_j=w_j/4\) recovers the sufficient half of
the Shannon-entropy theorem.  Taking only the \(w_0\mu_0\) term recovers a
compact null-shadow measure.  More importantly, the \(\mu_j\)'s may
themselves contain arbitrarily complicated countable or nonatomic
arrangements: entropy is charged only at the scale of separated clusters.
\end{remark}

\subsection{Compact supports and logarithmic potentials}

For a compact set \(K\subset\sphere^{p-1}\), define its spherical
orthogonality shadow by
\begin{equation}
  \mathcal O(K)
  :=
  \{u\in\sphere^{p-1}:a^\top u=0
    \text{ for some }a\in K\}.
  \label{eq:orthogonality-shadow}
\end{equation}

\begin{corollary}[Null orthogonality shadow]
\label{cor:null-shadow}
Let \(\mu\) be a Borel probability measure supported on a compact set
\(K\subset\sphere^{p-1}\).  If \(\mathcal O(K)\) is Haar-null, then the
conclusion of \cref{thm:measure-beta} holds.

Moreover, on the open cone
\begin{equation}
  \Omega_K
  :=
  \{x\in\R^p:\min_{a\in K}|a^\top x|>0\},
  \label{eq:shadow-free-cone}
\end{equation}
the continuous master function
\begin{equation}
  \Phi_\mu(x)
  :=
  \exp\left\{
    \int_K\log|a^\top x|\,\mu(\dd a)
  \right\}
  \label{eq:continuous-master-function}
\end{equation}
is smooth and positively homogeneous of degree one, and
\begin{equation}
  \nabla\log\Phi_\mu(x)=g_\mu(x).
  \label{eq:continuous-log-gradient}
\end{equation}
Consequently, for Haar-almost every \(u\), \(H_\mu(u)\) is both the squared
incidence cosine and the squared distance from the origin to the tangent
hyperplane of the level hypersurface of \(\Phi_\mu\) through \(u\).
\end{corollary}

\begin{proof}
For \(u\notin\mathcal O(K)\), compactness gives
\[
  \delta_K(u):=\min_{a\in K}|a^\top u|>0.
\]
It follows that \(I_\mu(u)\leq\delta_K(u)^{-1}<\infty\), so
\cref{thm:measure-beta} applies.

For \(x_0\in\Omega_K\), the same compactness argument gives a neighborhood
of \(x_0\) on which all denominators are uniformly bounded away from zero.
Differentiation under the integral in
\cref{eq:continuous-master-function} is therefore valid to every order and
gives \cref{eq:continuous-log-gradient}.  Since \(\mu\) is a probability
measure,
\[
  \Phi_\mu(tx)=t\Phi_\mu(x),
  \qquad t>0.
\]
Euler's identity gives \(x^\top g_\mu(x)=1\).  Thus the tangent hyperplane at
\(u\) is
\[
  \{x:g_\mu(u)^\top x=1\},
\]
whose squared distance from the origin is \(1/\|g_\mu(u)\|^2\).
\end{proof}

\begin{corollary}[Zero-length supports at the critical endpoint]
\label{cor:subcritical-support}
Under the notation of \cref{cor:null-shadow}, if
\begin{equation}
  \mathcal H^1(K)=0,
  \label{eq:subcritical-support}
\end{equation}
then \(\mathcal O(K)\) is Haar-null and all conclusions of
\cref{cor:null-shadow} hold.
\end{corollary}

\begin{proof}
For every \(\epsilon>0\), cover \(K\) by chordal balls
\(B(a_i,r_i)\) with arbitrarily small radii and
\(\sum_i r_i<\epsilon\).  If \(u\in\mathcal O(K)\), choose
\(a\in K\) with \(a^\top u=0\) and then an \(i\) with
\(\|a-a_i\|<r_i\).  It follows that
\[
  |a_i^\top u|\leq\|a_i-a\|<r_i.
\]
The Haar measure of a spherical belt
\(\{u:|a_i^\top u|<r_i\}\) is at most \(C_pr_i\), uniformly for small
\(r_i\).  Therefore
\[
  \sigma\{\mathcal O(K)\}\leq C_p\sum_i r_i<C_p\epsilon.
\]
Letting \(\epsilon\downarrow0\) proves the claim.
\end{proof}

\begin{theorem}[Rectifiable critical obstruction]
\label{thm:rectifiable-obstruction}
Let \(\gamma:J\to\sphere^{p-1}\) be a nonconstant regular \(C^1\)
immersion of an interval.  Suppose that \(\mu\) has a nonzero
component along this arc: for some \(h\geq0\), positive on a set of positive
Lebesgue measure,
\[
  \mu\geq\gamma_\#\{h(s)\,\dd s\}.
\]
Then
\begin{equation}
  \sigma\{u:I_\mu(u)=\infty\}>0.
  \label{eq:rectifiable-divergence}
\end{equation}
\end{theorem}

\begin{proof}
Let \(E\) be the positive-measure set of Lebesgue points \(s\) at which
\(h(s)>0\).  Localizing if necessary, choose a compact subinterval on
which \(E\) still has positive measure and \(\gamma\) is Lipschitz.
Consider the incidence manifold
\[
  \mathcal M
  :=
  \{(s,u)\in J\times\sphere^{p-1}:\gamma(s)^\top u=0\}.
\]
The coarea formula in local incidence charts shows that the part of
\(\mathcal M\) above \(E\) has positive \((p-1)\)-dimensional measure.
For each \(s\), the exceptional fiber on which
\(\gamma'(s)^\top u=0\) has codimension one inside
\(\{u:\gamma(s)^\top u=0\}\).  Thus the part of \(\mathcal M\) above
\(E\) where \(\gamma'(s)^\top u\neq0\) has positive
\((p-1)\)-dimensional measure.  At precisely those transverse pairs, the
projection \((s,u)\mapsto u\) has full rank.  The area formula
\cite{mattila-1995}, applied on the compact localized chart where the
projection is Lipschitz, therefore shows that its image has positive
spherical Haar measure.  Injectivity of \(\gamma\) is not needed: the area
formula counts multiplicity.

For every \(u\) in that image, choose a transverse \(s\in E\).  Locally,
\[
  |\gamma(t)^\top u|\leq C|t-s|.
\]
Since \(s\) is a positive Lebesgue point of \(h\), the integral
\(\int_{|t-s|<\delta} h(t)|t-s|^{-1}\dd t\) diverges; for example, each sufficiently
small dyadic annulus contributes an amount bounded away from zero.
Consequently
\[
  I_\mu(u)
  \geq
  \int_{|t-s|<\delta}
  \frac{h(t)}{|\gamma(t)^\top u|}\,\dd t
  =\infty,
\]
which proves \cref{eq:rectifiable-divergence}.
\end{proof}

\begin{remark}[Both behaviors occur in dimension one]
\label{rem:critical-support}
Since \(\dim_{\mathrm H}K<1\) implies \(\mathcal H^1(K)=0\),
\cref{cor:subcritical-support} strictly strengthens the earlier
subcritical dimension criterion.  It also reaches dimension one: a Moran
Cantor set in an arc with \(2^n\) stage-\(n\) intervals of length
\(2^{-n}/n\) has Hausdorff dimension one and zero \(\mathcal H^1\)-measure,
so every probability measure on it is admissible.  In the opposite
direction, \cref{thm:rectifiable-obstruction} rules out every nonzero
absolutely continuous component on a regular arc.

For the most symmetric negative example, let
\(L\subset\R^p\) be a two-dimensional subspace and let
\[
  K=L\cap\sphere^{p-1}\cong\sphere^1.
\]
Then \(\dim_{\mathrm H}K=1\), but
\(\mathcal O(K)=\sphere^{p-1}\): every vector has a direction in \(L\)
orthogonal to its projection onto \(L\).

If \(\mu\) is normalized arc measure on \(K\), then
\(I_\mu(u)=\infty\) for every \(u\).  Indeed, when
\(b=\operatorname{proj}_L u\neq0\), a rotation within \(L\) gives
\[
  I_\mu(u)
  =
  \frac{1}{2\pi\|b\|}
  \int_0^{2\pi}\frac{\dd\theta}{|\cos\theta|}
  =\infty,
\]
while \(b=0\) makes every denominator vanish.  Thus ordinary absolute
integration cannot be replaced silently by principal values.

Although the shadow-free cone \(\Omega_K\) is empty in this example, the
logarithmic integral remains finite on the larger set
\(\{x:\operatorname{proj}_Lx\neq0\}\) and satisfies
\[
  \Phi_\mu(x)=\frac12\|\operatorname{proj}_Lx\|
\]
when \(\operatorname{proj}_Lx\neq0\).  Its natural symmetric
principal-value logarithmic gradient is
\[
  g_\mu^{\mathrm{pv}}(x)
  =
  \frac{\operatorname{proj}_Lx}
       {\|\operatorname{proj}_Lx\|^2}.
\]
It would give
\[
  \frac{1}{\|g_\mu^{\mathrm{pv}}(U)\|^2}
  =
  \|\operatorname{proj}_LU\|^2,
\]
which has law
\(\Beta(1,(p-2)/2)\) for \(p>2\), and is identically one for \(p=2\).
Neither is the law in \cref{eq:measure-beta}.  Absolute reciprocal
integrability is therefore structural, not merely a convenient way of
choosing a version of the field.
\end{remark}

\begin{remark}[Singularity and full support do not decide convergence]
\label{rem:full-support-counterexamples}
Choose a dense sequence of centers \(a_j\), radii
\(\rho_j=2^{-2j}\), compact Cantor sets
\(a_j\in K_j\subset B(a_j,\rho_j)\) with
\(\mathcal H^1(K_j)=0\), nonatomic
probabilities \(\mu_j\) with \(\operatorname{supp}\mu_j=K_j\), and weights
\(w_j=2^{-j}\).
Set \(\mu:=\sum_{j\geq1}w_j\mu_j\).
Then \cref{thm:clustered-integrability} makes \(\mu\) admissible.
The dense centers give full support, while the countable union of the
zero-\(\mathcal H^1\) carriers shows that \(\mu\) is singular; it is
nonatomic because every \(\mu_j\) is.

In the other direction, choose weights of infinite Shannon entropy and a
Haar-generic sequence of atoms.  Almost surely the atoms are dense, while
\cref{thm:entropy-threshold} gives \(I_\mu(U)=\infty\) almost everywhere.
This measure is atomic, singular, and also has full support.  Thus neither
singularity nor topological support, separately or together, characterizes
admissibility.
\end{remark}

\section{Equivalent forms and geometric interpretation}
\label{sec:equivalent-forms}

The universal law has several useful reformulations.  We record only those
that clarify the incidence mechanism.

\subsection{A coordinate Dirichlet identity}

Take \(k=p\) and \(a_j=e_j\).  If
\[
  S_j=U_j^2,
  \qquad
  (S_1,\ldots,S_p)\sim
  \Dir\!\left(\frac12,\ldots,\frac12\right),
\]
then
\[
  \|g_{a,w}(U)\|^2
  =
  \sum_{j=1}^p\frac{w_j^2}{S_j}.
\]
The finite incidence theorem therefore gives the following exact
weighted-harmonic-mean identity.

\begin{corollary}[Dirichlet reciprocal identity]
\label{cor:dirichlet-reciprocal}
For every simplex vector \(w\),
\[
  \left(\sum_{j=1}^p\frac{w_j^2}{S_j}\right)^{-1}
  \sim
  \Beta\!\left(\frac12,\frac{p-1}{2}\right).
\]
\end{corollary}

For active coordinates, the variable on the left is the weighted harmonic
mean of \(S_j/w_j\), with weights \(w_j\).  Thus every deterministic choice
of the averaging weights produces the law of a single coordinate energy,
even though all coordinates enter the harmonic mean.  This identity is the
orthogonal-coordinate specialization of the spherical theorem; it does not
by itself explain invariance under a nonorthogonal arrangement.

The same formula has an information-geometric form.  For probability vectors
\(w\) and \(s\) with positive coordinates, define the Pearson divergence
\[
  D_{\chi^2}(w\Vert s)
  :=\sum_{j=1}^p\frac{(w_j-s_j)^2}{s_j}
  =\sum_{j=1}^p\frac{w_j^2}{s_j}-1.
\]

\begin{corollary}[Pearson-divergence form]
\label{cor:pearson-divergence}
If
\(S\sim\Dir(1/2,\ldots,1/2)\), then, for every simplex vector \(w\),
\[
  D_{\chi^2}(w\Vert S)
  \sim
  \BetaPrime\!\left(\frac{p-1}{2},\frac12\right)
  \overset d=
  \frac{\chi_{p-1}^2}{\chi_1^2}
  =(p-1)F_{p-1,1}.
\]
Here the two chi-square variables are independent.
\end{corollary}

\begin{proof}
By \cref{cor:dirichlet-reciprocal},
\(\{1+D_{\chi^2}(w\Vert S)\}^{-1}\) has the stated Beta law.  The displayed
Beta-prime and chi-square forms are the standard transformations of that
variable.
\end{proof}

\subsection{The projective polar map and the Cremona special case}

For strictly positive weights in the coordinate arrangement, the projective
direction of the logarithmic gradient is
\begin{equation}
  \mathcal C_w[x_1:\cdots:x_p]
  = [w_1/x_1:\cdots:w_p/x_p].
  \label{eq:weighted-cremona}
\end{equation}
This is a diagonally weighted version of the standard Cremona transformation:
it is birational and \(\mathcal C_w\circ\mathcal C_w\) is the identity
wherever both sides are defined.  The unit vector \(N_{e,w}(u)\) is the spherical representative
of \(\mathcal C_w[u]\), so \cref{thm:universal-beta} becomes an incidence law
between a Haar ray and its weighted Cremona image.  Polar maps of homogeneous
functions and products of linear forms provide the surrounding algebraic
geometry \cite{dolgachev-2000,cohen-denham-falk-varchenko-2012}.

The terminology has a limited but legitimate scope.  If \(k=p\) and the
arrangement matrix \(A=[a_1,\ldots,a_p]\) is invertible, then projectively
\[
  [x]
  \longmapsto
  [A^\top x]
  \longmapsto
  [w/(A^\top x)]
  \longmapsto
  [A\{w/(A^\top x)\}]
  =[g_{a,w}(x)]
\]
is a Cremona map composed with projective linear changes of coordinates.  For
dependent or overcomplete arrangements, the logarithmic-gradient map remains
a rational map after denominators are cleared but need not be a Cremona
transformation.  In neither case does this algebraic description imply
preservation of spherical Haar or
projective volume; it identifies the map, not the probability mechanism.

\subsection{Conormal, tangent-plane, and log-barrier interpretations}

The tangent-plane calculation in \cref{sec:incidence-object} says that the
reciprocal field is an Euler-normalized centro-affine conormal, up to the
conventional choice of transversal sign, and that the incidence variable is
the squared origin-to-tangent-plane distance.  Arrangement master functions
provide the logarithmic-gradient background
\cite{cohen-denham-falk-varchenko-2012}; for the centro-affine terminology,
see \cite{nomizu-pinkall-1987}.

Positivity has an additional convex interpretation.  On a chamber choose
signs \(\varepsilon_j\) so that
\(\varepsilon_j a_j^\top x>0\), and define the logarithmic barrier
\[
  \mathcal B(x)
  :=-\sum_{j=1}^k w_j
       \log(\varepsilon_j a_j^\top x).
\]
Then
\[
  g_{a,w}(x)=-\nabla\mathcal B(x),
  \qquad
  \nabla^2\mathcal B(x)
  =\sum_{j=1}^k w_j
    \frac{a_ja_j^\top}{(a_j^\top x)^2}\succeq0.
\]
Thus the same vector is a conormal to a master-function level set and the
negative gradient of a convex chamber barrier.  This explains geometrically
why positive weights form the natural class, although it does not by itself
prove the Beta law.

None of these descriptions asserts that \(u\mapsto N(u)\) preserves Haar
measure, affine volume, or projective volume.  The theorem identifies one
scalar pushforward only.  The swirl example in
\cref{ex:swirl-counterexample} makes this limitation concrete.

\section{Discussion and open problems}
\label{sec:discussion}

The results distinguish three levels of information.  The
unconditional Beta marginal determines exactly the tangent-norm law but
does not determine a gradient.  Plane-conditional Cauchy laws recover the
combined positive coefficients of every fixed finite reciprocal
arrangement, and recover magnitudes but not measurable signs in the zonal
class.  Adding an oriented Hardy condition removes that sign ambiguity.
For signed reciprocal measures, the conditional laws make every phase
pushforward positive and expose an exact cancellation deficit equal to twice
the negative mass.  A phase-norming condition then gives local-to-global
positivity without requiring injectivity.  Only opposite-sign collisions
matter: a cross-sign product-dimension criterion is sufficient, and negative
atoms are ruled out with no separation assumption.  The classical signed
boundary-tail input is what turns each one-dimensional Cauchy projection into
positivity before projective tomography addresses the remaining diffuse
cancellation.

On the infinite side, \cref{thm:wiener-incidence} reduces absolute
integrability exactly to a projection Dini series.
\cref{thm:entropy-threshold,thm:clustered-integrability} show why neither
weights nor geometry alone can classify all measures.  The critical
endpoint is already nontrivial: zero \(\mathcal H^1\)-length is sufficient,
while a rectifiable component is obstructive.

The remaining problems are consequently narrower than the initial
classification questions.

\begin{enumerate}
\item Find intrinsic energy or capacity conditions equivalent to the
  almost-everywhere Wiener series in
  \cref{eq:wiener-profile-series}.  At dimension one, determine the boundary
  between the zero-length result in \cref{cor:subcritical-support} and the
  rectifiable obstruction in \cref{thm:rectifiable-obstruction}, especially
  for purely unrectifiable carriers of positive finite
  \(\mathcal H^1\)-measure.
\item Decide whether reciprocal integrability and plane-conditional Cauchy
  laws force the phase deficit in \cref{eq:phase-norm-deficit} to vanish.
  Equivalently, can a nonzero diffuse negative part remain dominated by the
  positive part in almost every phase pushforward?  More generally, under an
  oriented analytic hypothesis, determine when a compatible family of inner
  Cayley boundary maps must arise from one positive reciprocal measure on
  projective space.  The sign family in \cref{eq:zonal-sign-family} shows that
  no unoriented version can hold.
\item Characterize the singular circle measures satisfying the ordinary
  secant Dini condition in \cref{eq:secant-integrability}.  Clark theory
  alone controls nontangential boundary values and does not supply this
  absolute-integrability property.
\end{enumerate}

\begin{acks}[AI-use disclosure]
During the development of this manuscript, OpenAI GPT-5.6 Sol, operating
through Codex, assisted with literature searching, mathematical exploration,
and writing.
Anthropic Claude Opus 5 was used for independent proof audits and critical
manuscript review.  The author made all substantive decisions and accepts
full responsibility for the mathematical claims, citations, and text.
\end{acks}

\bibliographystyle{\BetaIncidenceBibliographyStyle}
\bibliography{beta-incidence-angle}
 
\end{document}